\documentclass[reqno]{amsart}
\usepackage{amsmath}
\usepackage{amssymb}
\usepackage{amsthm}
\usepackage[abbrev]{amsrefs}
\usepackage{color}
\usepackage{xcolor}
\usepackage{amsfonts}
\usepackage{bm}

\usepackage{graphicx}
\usepackage{enumerate}
\usepackage{url}
\usepackage{bm}

\usepackage{float}
\usepackage{setspace}
\usepackage{comment}

\theoremstyle{definition}
\newtheorem{thm}{Theorem}[section]

\newtheorem{lem}[thm]{Lemma}

\newtheorem{prop}[thm]{Proposition}

\newtheorem{ex}[thm]{Example}

\newtheorem{defi}[thm]{Definition}

\theoremstyle{remark}
\newtheorem{rem}{Remark}[section]

\newtheorem*{acknowledgment}{Acknowledgments}

\numberwithin{equation}{section}
\numberwithin{figure}{section}

\newcommand{\End}{\mathrm{End}}
\newcommand{\Hom}{\mathrm{Hom}}
\newcommand{\bigslant}[2]{{\raisebox{.2em}{$#1$}\left/\raisebox{-.2em}{$#2$}\right.}}
\def\e{{\epsilon}}
\def\co{\colon\thinspace}

\title{The Heisenberg double of involutory Hopf algebras and invariants of closed $3$-manifolds}

\author{Serban Matei Mihalache}%
\address{Department of Mathematics, Tohoku University, 
6-3, Aoba, Aramaki-aza, Aoba-ku, 
Sendai, 980-8578, Japan}
\email{matei.mihalache.q3@dc.tohoku.ac.jp}

\author{Sakie Suzuki}%
\address{Department of Mathematical and Computing Science, School of Computing,
Tokyo Institute of Technology,
2-12-1 Ookayama, Meguro-ku, Tokyo 152-8552, Japan}
\email{sakie@c.titech.ac.jp}

\author{Yuji Terashima}%
\address{Department of Mathematics, Tohoku University, 
6-3, Aoba, Aramaki-aza, Aoba-ku, 
Sendai, 980-8578, Japan}
\email{yujiterashima@tohoku.ac.jp}

\begin{document}
\maketitle

\begin{abstract}
We construct an invariant of closed oriented $3$-manifolds using a finite dimensional, involutory, unimodular and counimodular Hopf algebra $H$. We use the framework of normal o-graphs introduced by R. Benedetti and C. Petronio, in which one can represent a branched ideal triangulation via an oriented virtual knot diagram. 
We assign a copy of a canonical element of the Heisenberg double $\mathcal{H}(H)$ of $H$ to each real crossing, which represents a branched ideal tetrahedron.
The invariant takes values in the cyclic quotient $\mathcal{H}(H)/{[\mathcal{H}(H),\mathcal{H}(H)]}$, which is isomorphic to the base field.
In the construction we use only the canonical element and  structure constants of $H$ and we do not use any representations of $H$. This, together with the finiteness and locality conditions of the moves for normal o-graphs,  makes the calculation of our invariant  rather simple and easy to understand. When $H$ is the group algebra of a finite group, the invariant counts the number of group homomorphisms from the fundamental group of the $3$-manifold to the group.
\end{abstract} 

\tableofcontents

\section{Introduction}
S. Baaj and G. Skandalis \cite{BS} and R.M. Kashaev \cite{Ka} found a striking fact that for the Heisenberg double $\mathcal{H}(H)$ of any finite dimensional Hopf algebra $H$, there exists a canonical element $T$ in $\mathcal{H}(H)^{\otimes 2}$ satisfying 
the pentagon equation
\begin{align*}
T_{12}T_{13}T_{23}=T_{23}T_{12}.
\end{align*}
In \cite{Ka}, he also showed that the Drinfeld double $\mathcal{D}(H)$ of $H$ can be realized as a subalgebra of  $\mathcal{H}(H)^{\otimes 2}$, and observed that the universal $R$-matrix of $\mathcal{D}(H)$
can be represented as a combination of four copies of $T$, where  the quantum Yang-Baxter equation of the universal $R$-matrix follows from a sequence of the pentagon equation of $T$.
Using his results,  the second author \cite{S} reconstructed the universal quantum $\mathcal{D}(H)$ invariant of framed tangles by assigning a copy of the canonical element $T$ to each branched ideal tetrahedron of the tangle complements, and expected that this construction leads to invariants of pairs of a $3$-manifold and geometrical input. In the present paper, we show that this construction defines an invariant of closed oriented $3$-manifolds when the Hopf algebra $H$ is involutory, unimodular and counimodular.

In the formulation of our invariant, we use a diagrammatic representation of closed oriented $3$-manifolds introduced by R. Benedetti and C. Petronio  \cite{BP}.
Their diagrams, which are called closed normal o-graphs, are oriented virtual knot diagrams satisfying certain conditions. 
They showed that homeomorphism classes of closed oriented $3$-manifolds are identified with equivalence classes of closed normal o-graphs up to certain moves. A crossing of a closed normal o-graph represents a branched ideal tetrahedron in the corresponding $3$-manifold, and the orientation of the strand specifies a way to extend these local branching structures to a global one.
Our invariant is obtained by assigning a copy of canonical element $T$ (or its inverse) to each crossing of closed normal o-graphs,
and by reading them along the strands.
The invariant takes values in the cyclic quotient $\mathcal{H}(H)/{[\mathcal{H}(H),\mathcal{H}(H)]}$,  which is isomorphic to the base field through the character of the Fock space representation. 

The proof of the invariance will be performed by checking the invariance under each move for normal o-graphs.
There are two important types: the MP-moves and the  CP-move.
An MP-move represents a Pachner move equipped with a branching structure, which corresponds to a modified pentagon equation in the level of the invariant.
Here, the modification is related to the antipode of $H$, and we need to assume that the antipode is involutive.
Up to the MP-moves (and the 0-2 move),  a closed normal o-graph represents a closed oriented $3$-manifold with a combing.
The CP-move is a special move for branched triangulations, which changes the combing. The invariance under the CP-move will be shown using tensor networks, where we need to assume that $H$ is in addition both unimodular and counimodular. 
Using tensor networks we will also show a connected sum formula of the invariant.

There are several invariants based on triangulations and involutory Hopf algebras \cites{TV, BW, Ku1}.
Our invariant uses only the canonical element and  the structure constants of Hopf algebras and does not use any representations, and thus the construction is rather simple.
Furthermore, the moves for closed normal o-graphs are local and finite, unlike  handle slide moves in the Kirby calculus for link surgery presentations of  $3$-manifolds, and each of these finite moves corresponds to an algebraic equation nicely, like the MP-moves and the pentagon equation. Thus the proof of the invariance is also rather easy to understand.

The main example of finite dimensional, involutory, unimodular and counimodular Hopf algebras is the group algebra $\mathbb{C}[G]$ of a finite group $G$. We show that in this case the invariant is same as the number of homomorphisms from the fundamental group of the $3$-manifold to $G$. 
There are several other examples of Hopf algebras which would be interesting to be used. The restricted enveloping algebras of restricted Lie algebras \cite{J} are finite dimensional involutory Hopf algebras.
In \cite{MP}, S. Majid and A. Pachol classified Hopf algebras under dimension $\leq$ 4 over the field of characteristic $2$.
M. Kim \cite{Ki} also gave some examples of finite dimensional, involutory unimodular and counimodular  (commutative and cocommutative) Hopf algebras  which are not group algebras.

The rest of the paper is organized as follows. In Section \ref{sec:hd}, we discuss Hopf algebras and the Heisenberg double of them. 
In Section \ref{sec:ce}, following \cite{BP}, we explain how to represent closed oriented $3$-manifolds in a combinatorial manner using closed normal o-graphs.
In Section \ref{sec:ri}, we explain the construction of our invariant $Z(M;\mathcal{H}(H))=Z(\Gamma ;\mathcal{H}(H))$ using the Heisenberg double  $\mathcal{H}(H)$  and a closed normal o-graph $\Gamma$ which represents a $3$-manifold $M$.
The proof of the invariance will be given in two sections.
In Section \ref{sec:mt}, we prove the invariance  under all moves except for the CP-move.
In Section \ref{sec:tnm}, we reformulate our invariant using tensor networks, and  by using them 
we prove the  invariance  under the CP-move.
In Section \ref{sec:csf}, we show the connected sum formula, and study the case for the group algebra $\mathbb{C}[G]$ of a finite group $G$.

\begin{acknowledgment}
 We would like to thank S. Baseilhac, R. Benedetti, K. Hikami, M. Ishikawa, R.M. Kashaev, A. Kato, Y. Koda, T.T.Q. L$\hat{\mathrm{e}}$ for valuable discussions. This work is partially supported by
 JSPS KAKENHI Grant Number JP17K05243, JP19K14523, and by JST CREST Grant Number JPMJCR14D6.
\end{acknowledgment}

\section{Hopf algebra and Heisenberg double}
\label{sec:hd}
In this section, we quickly review the definition and some properties of the Heisenberg double of Hopf algebras. 

\subsection{Hopf algebra}\label{Hopf}
A \textit {Hopf algebra} $H$ over a field $\mathbb{K}$ is a vector space equipped with five linear maps
\begin{align*}
&M\co H\otimes H \to H, \quad
1\co \mathbb{K} \to H, \quad
\Delta\co H  \to H\otimes H, \quad
\epsilon\co H \to \mathbb{K}, \quad
S\co H \to H,
\end{align*}
called multiplication, unit, comultiplication, counit and antipode respectively, satisfying the standard axioms of Hopf algebras.
When the antipode is involutive, i.e., $S^2 = \mathrm{id}_H$, we call $H$  \textit{involutory}.
Throughout the paper, $H$ will denote a finite dimensional Hopf algebra and $H^*$ will denote the \textit {dual Hopf algebra} of $H$.
We will also use the Sweedler notation $\Delta(x)=x_{(1)}\otimes x_{(2)}$ for $x\in H$. 

Recall that a \textit {right integral}
of a Hopf algebra $H$ is an element $\mu_{R} \in H^*$ satisfying $\mu_{R}\cdot f=\mu_{R}f(1)$ for every $f\in H^*$.
A \textit {left integral} is defined similarly.
Since $H$ is finite dimensional, a left (resp. right) integral of the dual Hopf algebra $H^*$ is an element of $H$ and is called a \textit {left (resp. right) cointegral} of $H$.
It is well known that for a finite dimensional Hopf algebra, an integral always exists and is unique up to scalar multiplication.
We say $H$ is \textit {unimodular} when the left cointegrals are also right cointegrals, and \textit {counimodular} when the left integrals are also the right integrals.
For more details on Hopf algebras and its integrals, see \cites{R, Ku1}.

\subsection{Heisenberg double}
We use the left action of $H$ on $H^*$ defined by $(a\rightharpoonup f)(x):=f(xa)$, for $a, x\in H$ and $f\in H^*$.
The \textit{Heisenberg double}
\begin{align*}
\mathcal{H}(H)=H^{*}\otimes H
\end{align*}
of a Hopf algebra $H$ is a $\mathbb{K}$-algebra with unit $\epsilon \otimes 1$ and multiplication given by
\begin{gather*}
(f\otimes a)\cdot (g\otimes b)=f\cdot (a_{(1)}\rightharpoonup g)\otimes a_{(2)}b,
\end{gather*}
for $a,b\in H$ and $f,g\in H^*$.

Let $\{e_i\}$ be the basis of $H$ and $\{e^i\}$ its dual basis.
Then the canonical element is given by
\begin{align}
T=\sum_i(\epsilon \otimes e_i)\otimes (e^i\otimes 1) \quad \in \mathcal{H}(H)^{\otimes 2} \label{eq:can}
\end{align}
and its inverse by
\begin{align}
\overline{T}=\sum_i(\epsilon \otimes S(e_i))\otimes (e^i\otimes 1) \quad \in \mathcal{H}(H)^{\otimes 2}. \label{eq:can inv}
\end{align}

In the case of the Drinfeld double $\mathcal{D}(H)$ of $H$, the canonical element satisfies the quantum Yang-Baxter equation, which in turn produces invariants of links and 3-manifolds \cites{H, KR}. One important feature of the Heisenberg double which plays a central role in our construction of invariants is that the canonical element satisfies the pentagon equation.
\begin{prop}(\cites{BS, Ka})\label{pentagon}
The pentagon equation 
\begin{align}\label{eq:pentagon equation}
T_{12}T_{13}T_{23}=T_{23}T_{12}
\end{align}
holds in $\mathcal{H}(H)^{\otimes 3}$.
\end{prop}
\begin{proof}
Note that
\begin{align*}
T_{12}T_{13}T_{23}&=\sum\limits_{i,j,k}(\epsilon\otimes e_i)(\epsilon\otimes e_j)\otimes(e^i\otimes 1)(\epsilon\otimes e_k)\otimes(e^j\otimes 1)(e^k\otimes 1)
\\
&=\sum\limits_{i,j,k}(\epsilon\otimes e_i e_j)\otimes(e^i\otimes e_k)\otimes(e^j e^k\otimes 1)\quad \in \epsilon\otimes (H\otimes H^*)^{\otimes 2}\otimes 1,
\\
T_{23}T_{12}&=\sum\limits_{i,j}(\epsilon\otimes e_j)\otimes(\epsilon\otimes e_i)(e^j\otimes 1)\otimes(e^i\otimes 1)\\
&=\sum\limits_{i,j}(\epsilon\otimes e_j)\otimes(e_{i(1)}\rightharpoonup e^j\otimes e_{i(2)})\otimes(e^i\otimes 1) \quad  \in \epsilon\otimes (H\otimes H^*)^{\otimes 2}\otimes 1.
\end{align*}
Let us identify $(H\otimes H^*)^{\otimes 2}$ with $\text{End}(H\otimes H)$ through the map
\begin{align*}
& \iota\colon x\otimes f\otimes y\otimes g\mapsto(a\otimes b\mapsto f(a)x\otimes g(b)y),
\end{align*}
for $x,y,a,b \in H$ and $f,g\in H^*$.
Then, after identifying between $\epsilon\otimes (H\otimes H^*)^{\otimes 2}\otimes 1$ and $(H\otimes H^*)^{\otimes 2}$, we can see that the both elements $T_{12}T_{13}T_{23}$ and $T_{23}T_{12}$ are sent by $ \iota $ to 
the same element as follows.
\begin{align*}
\iota \left(T_{12}T_{13}T_{23}\right)(a\otimes b)&=e^i(a)e_ie_j\otimes e^je^k(b)e_k
\\
&=ae_j\otimes e^j(b_{(1)})e^k(b_{(2)})e_k
\\
&=ab_{(1)}\otimes b_{(2)}, 
\\
 \iota (T_{23}T_{12}) (a\otimes b)
&= e_j (e_{i(1)}\rightharpoonup e^j)(a)\otimes e^i(b)e_{i(2)}
\\
&=e_je^j(ae_{i(1)})\otimes e^i(b)e_{i(2)}
\\
&=ae_{i(1)}\otimes e^i(b)e_{i(2)}
\\
&=ab_{(1)}\otimes b_{(2)}.
\end{align*}
Thus we have the assertion.
\end{proof}

The Heisenberg double $\mathcal{H}(H)$ has a canonical module $F(H^{*})=H^*$, which we call the \textit{Fock space}, with the  action $\phi:\mathcal{H}(H)\to\End(H^*)$ given by
\begin{align}
\phi(f\otimes a)(g)=(f\otimes a)\triangleright g:=f\cdot (a\rightharpoonup g) \label{eq:fock}
\end{align}
for $(f\otimes a) \in \mathcal{H}(H)$ and $g\in H^*$. 
Let $\raisebox{2pt}{$\chi$}_{Fock}$ be the character associated to the Fock space.
For a $\mathbb{K}$-algebra $A$, let $[A,A]$ be the subspace spanned by $\{ xy-yx \mid x, y \in A\}$ over $\mathbb{K}$.
We are interested in the quotient space ${\mathcal{H}(H)}/{[\mathcal{H}(H),\mathcal{H}(H)]}$ since this is the space where the invariant takes values in.

\begin{prop}\label{prop:K}
The character of the Fock space
\begin{align*}
\raisebox{2pt}{$\chi$}_{Fock}:\bigslant{\mathcal{H}(H)}{[\mathcal{H}(H),\mathcal{H}(H)]}\to \mathbb{K}
\end{align*}
is an isomorphism between vector spaces.
\end{prop}

\begin{proof}
In  \cite{L}*{Proposition 6.1} it was shown that 
for $F \in \End(H^*)$, the element
\begin{align*}
\sum_{i,j}F(e^i)e^j\otimes S^{-1}(e_j)e_i \in \mathcal{H}(H)
\end{align*}
is in the preimage of $F$ by $\phi$, i.e., 
for any $g \in H^*$ we have,
\begin{align*}
\sum_{i,j}\phi(F(e^i)e^j\otimes S^{-1}(e_j)e_i)(g) &= \sum_{i,j}F(e^i)e^j\cdot(S^{-1}(e_j)e_i\rightharpoonup g)\\
&= \sum_{i,j}g_{(3)}(e_i)\cdot F(e^i)\cdot g_{(2)}(S^{-1}(e_j))e^j\cdot g_{(1)}\\
&= F(g_{(3)})\cdot S^{-1}(g_{(2)})\cdot g_{(1)}\\
&=F(g_{(2)})\epsilon(g_{(1)})\\
&= F(g).\\
\end{align*}
Since dim$\mathcal{H}(H)=\text{dim}\text{End}(H^*)$, it follows that $\phi$ is bijective, hence $\phi$ is an algebra isomorphism.

Note that $\text{End}(H^*)$ is a matrix algebra, thus the canonical trace
\begin{align*}
\text{tr}\co\bigslant{\text{End}(H^*)}{[\text{End}(H^*),\text{End}(H^*)]}\to \mathbb{K}
\end{align*}
is also an isomorphism. Thus we have the assertion.
\end{proof}

When the antipode $S$ is involutive, $\raisebox{2pt}{$\chi$}_{Fock}$ can be given in terms of integrals.
Let $\mu_{R} \in H^*$ and $e_{L} \in H$ be a right integral and a left cointegral satisfying $\mu_{R}(e_{L})=1$.
\begin{prop}\label{prop:fock is int}
For an involutory Hopf algebra $H$, we have
\begin{align*}
\raisebox{2pt}{$\chi$}_{Fock}(f\otimes a)&=f(e_L)\mu_R(a)
\end{align*}
for $f\otimes a\in\mathcal{H}(H)$.
\end{prop}
\begin{proof}
For $F \in \End(H^*)$, the trace map is given (cf. \cite{R}*{Chap. 10}) by
\begin{align*}
\mathrm{tr}(F)=\langle e_L,F(\mu_{R(2)})S(\mu_{R(1)})\rangle.
\end{align*}
Thus we have
\begin{align*}
\raisebox{2pt}{$\chi$}_{Fock}(f\otimes a)&=\langle e_L, f(a\rightharpoonup\mu_{R(2)})S(\mu_{R(1)})\rangle\\
&=\langle e_L, \mu_{R(3)}(a)f\cdot\mu_{R(2)}\cdot S(\mu_{R(1)})\rangle\\
&=f(e_L)\mu_R(a).\\
\end{align*}
The third equality follows from the fact that $x_{(2)}S(x_{(1)})=\e(x)1$ for an involutory Hopf algebra.
\end{proof}

\section{Closed normal o-graph}
\label{sec:ce}
In order to define an invariant, we first recall the method introduced in \cite{BP} by R. Benedetti and C. Petronio to represent closed oriented 3-manifolds in a combinatorial manner.
This method is based on the theory of branched spines, which are the dual of branched ideal triangulations.

\begin{figure}[b]
  \centering
  \includegraphics{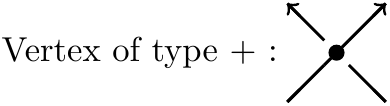}
  \hspace{10mm}
  \includegraphics{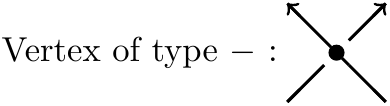}
  \caption{}
  \label{fig:vertex}
\end{figure}

\begin{defi}[\cite{BP}]\textbf{}
\label{def:bp}
A \textit{closed normal o-graph} is an oriented virtual knot diagram, i.e., 
a finite connected $4$-valent graph $\Gamma$ immersed in $\mathbb{R}^2$ with the following conditions:
\begin{spacing}{0.5}
\end{spacing}
\begin{description}
\setlength{\itemsep}{1mm}
\setlength{\parskip}{1mm} 
\item[N1] At each vertex, a sign $+$ or $-$ is indicated, which is represented by the over-under notation as in Figure \ref{fig:vertex}, 
\item[N2] Each edge has an orientation such that it matches among two edges which are opposite to each other at a vertex,
\item[C1] If one removes the vertices and joins the edges which are opposite to each other, the result is a unique oriented circuit,
\end{description}
satisfying the following additional conditions:
\begin{description}
\setlength{\itemsep}{1mm}
\setlength{\parskip}{1mm} 
\item[C2] The trivalent graph obtained from $\Gamma$ by the rules of \cite{BP}, page 7. Figure 1.2. is connected,
\item[C3] Consider the disjoint union of oriented circuits obtained from $\Gamma$ by the rules of \cite{BP}, page 7. Figure 1.3. Then the number of these circuits is exactly one more than the number of vertices of $\Gamma$.
\end{description}
\end{defi}

Let $\mathcal{G}$ be the set of closed normal o-graphs and $\mathcal{M}$ the set of oriented closed 3-manifolds up to orientation preserving homeomorphisms. 
Given a closed normal o-graph $\Gamma\in\mathcal{G}$, one can canonically construct a $3$-manifold $\Phi(\Gamma) \in\mathcal{M}$ as follows.
We fix an orientation of $\mathbb{R}^3$ and place a closed normal o-graph on $\mathbb{R}^2\subset \mathbb{R}^3$.
Then we replace each of its vertices with a tetrahedron (with the orientation given by $\mathbb{R}^3$), and glue the faces of the ideal tetrahedra. The way to glue the faces of ideal tetrahedra is specified by the order of vertices of ideal tetrahedra defined as in Figure \ref{fig:reconstruction map}, i.e.,
we glue faces by the orientation reversing map which preserves the order of vertices.
Conditions on closed normal o-graphs ensure that the result is an ideally triangulated 3-manifolds with $S^2$ boundary. 
Then, after we cap the boundary, the result defines an element $\Phi(\Gamma) \in\mathcal{M}$.
Here, for the geometrical meaning of the order of the vertices of ideal tetrahedra, see Remark \ref{branch}, where the meaning of the technical conditions \textbf{C1} \textbf{C2} and \textbf{C3} are also explained.
We denote the construction map obtained in the above way by $\Phi\co \mathcal{G} \to \mathcal{M}$. 

\begin{figure}
    \centering
    \includegraphics{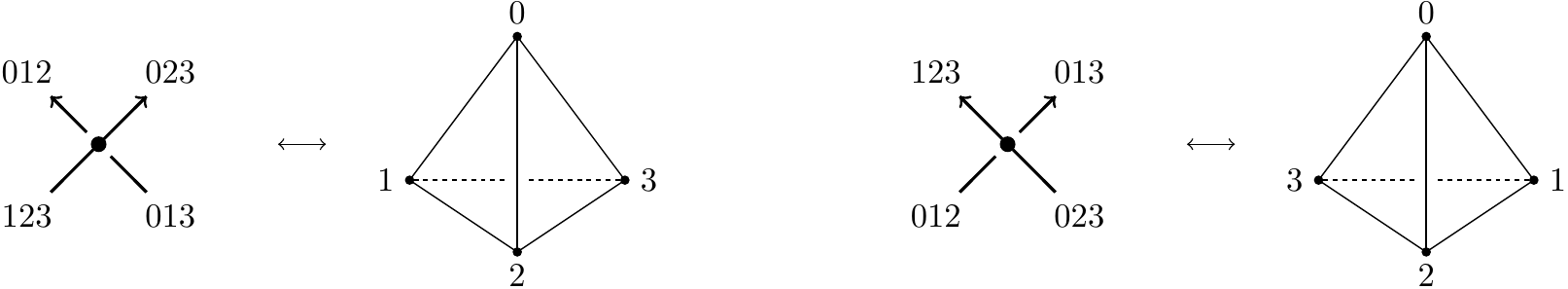}
    \caption{}
    \label{fig:reconstruction map}
\end{figure}

The map $\Phi$ is not one to one, and in order to make it so, we need the following local moves of the diagrams.
\begin{description}
\setlength{\itemsep}{1mm}
\setlength{\parskip}{1mm} 
\item[1]Planer isotopy of the diagram and the \textit{Reidemeister type moves} described in Figure \ref{fig:RM}.
\item[2]The \textit{0-2 move} and the \textit{MP-moves} described in Figure \ref{fig:02} and Figure \ref{fig:MP}, respectively.
\item[3]The \textit{CP-move} in Figure \ref{fig:CP}.
\end{description}

\begin{figure}
  \centering
  \begin{minipage}{0.62\columnwidth}
    \centering
    \includegraphics{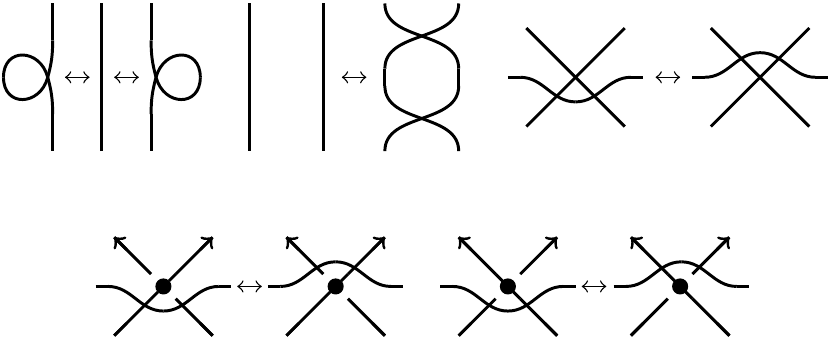}
    \caption{The Reidemeister type moves.}
    \label{fig:RM}
  \end{minipage}
  \begin{minipage}{0.37\columnwidth}
    \centering
    \includegraphics{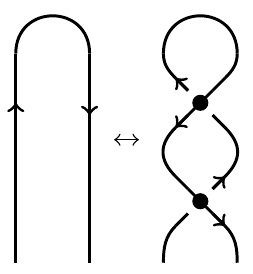}
    \caption{The 0-2 move.}
    \label{fig:02}
  \end{minipage}
\end{figure}

\begin{figure}
  \centering
  \includegraphics{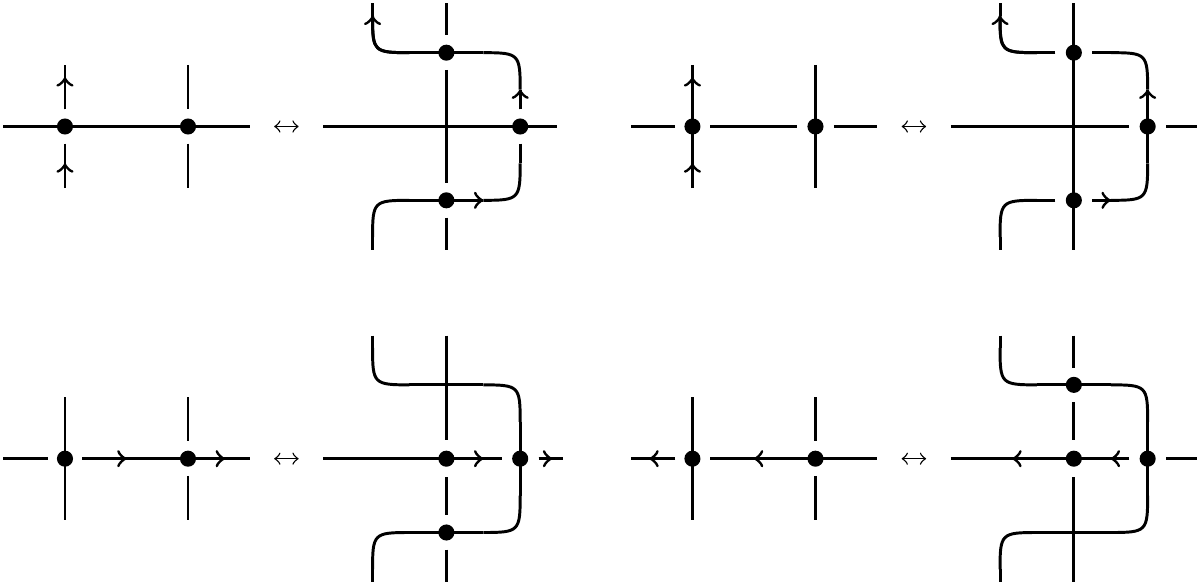}
  \caption{The MP-moves. The orientation of each non-oriented edge in the figure is arbitrarily if it matches before and after the move.}
  \label{fig:MP}
\end{figure}

\begin{figure}
  \centering
  \includegraphics{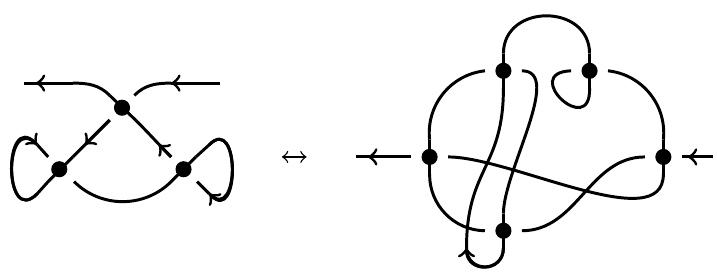}
  \caption{The CP-move.}
  \label{fig:CP}
\end{figure}

All of the above local moves preserve the axioms of closed normal o-graphs.
We say that two closed normal o-graphs are equivalent if one can be obtained from the other by planer isotopy and finite sequence of moves defined above.
Let us denote this equivalence relation by $\sim $.
The following was proved in \cite{BP}.

\begin{prop}
The map
\begin{align}
\Phi\co \bigslant{\mathcal{G}}{\sim}\to \mathcal{M}
\end{align}
is well-defined and one to one.
\label{thm:BP}
\end{prop}

\begin{ex}
The closed normal o-graph of the lens space $L(p,1)$, $p\geq 1$, is given by the following graph with $p$ vertices.
\begin{figure}[H]
  \centering
  \includegraphics{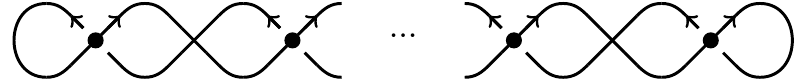}
  \label{lens_space}
\end{figure}
\end{ex}

\begin{rem}\label{branch}
We briefly remark on the geometrical meaning of the representation of $3$-manifolds by closed normal o-graphs; see \cite{BP} for more details. 
The order of vertices of ideal tetrahedra as in Figure \ref{fig:reconstruction map} specifies a branching structure for the ideal triangulation, which gives a combing, i.e., a non-vanishing vector field up to homotopy, to the underlying $3$-manifold.
The technical conditions \textbf{C1}, \textbf{C2} and \textbf{C3} ensure that the 3-manifold corresponding to a closed normal o-graph has a $S^2$ boundary with a nice branching structure, where the associated combing can be extended canonically to the closed $3$-manifold after we cap $B^3$ on the boundary. 
In this case, a closed normal o-graph up to the 0-2 move and the MP-moves represents a 3-manifolds with a combing.
The CP-move in Figure \ref{fig:CP} changes the combing while preserving the underline 3-manifold, thus one gets the complete representation of $\mathcal{M}$ as in Proposition \ref{thm:BP}.
Here, the CP-move is an interpretation of the Pontryagin surgery in terms of branched standard spine; see \cite{BP}*{Chap. 6} for details. 
\end{rem}

\section{Invariant}
\label{sec:ri}
\begin{figure}
  \centering
  \begin{minipage}{0.61\columnwidth}
    \centering
    \includegraphics{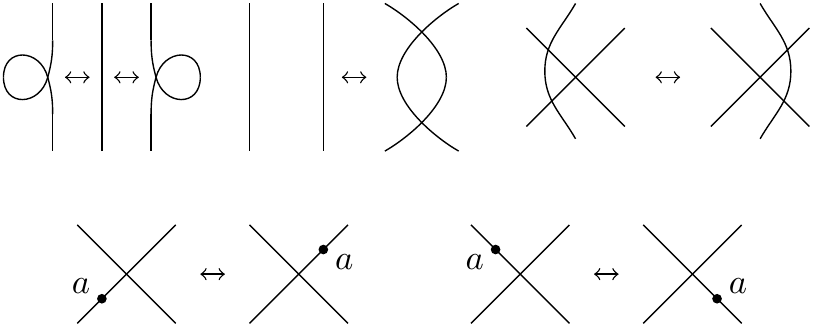}
    \caption{}
    \label{fig:crm}
  \end{minipage}
    \begin{minipage}{0.37\columnwidth}
    \centering
    \includegraphics{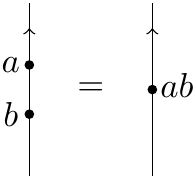}
    \caption{}
    \label{fig:bs}
  \end{minipage}
\end{figure}

\begin{figure}
  \centering
  \begin{minipage}{0.5\columnwidth}
    \centering
    \includegraphics{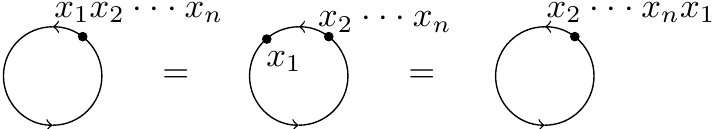}
    \caption{}
    \label{fig:bs2}
  \end{minipage}
\end{figure}

In this section,  we define a scalar $Z(\Gamma;\mathcal{H}(H))$ for a closed normal o-graph $\Gamma$.
In the following section, we show that this scalar is an invariant of  closed oriented 3-manifolds  when the Hopf algebra is involutory, unimodular and counimodular.

Let $A$ be a $\mathbb{K}$-algebra. \textit{$A$-decorated diagram} is an oriented closed curve immersed in $\mathbb{R}^2$ with the finite number of dots each of which is labeled with an element of $A$.
These dots are called beads.
We shall consider the $A$-decorated diagram up to planer isotopy and moves in Figure \ref{fig:crm} and \ref{fig:bs}.
We also allow the beads to slide along the curve.

We define the scalar $Z(\Gamma;\mathcal{H}(H))$ as follows.
Recall the definition of the canonical element $T$ and its inverse $\overline{T}$ given in (\ref{eq:can}) and (\ref{eq:can inv}).
Using the Sweedler notation, we write the canonical element as $T=T_1\otimes T_2\in\mathcal{H}(H)^{\otimes 2}$ and its inverse as $\overline{T}=\overline{T}_1\otimes \overline{T}_2\in\mathcal{H}(H)^{\otimes 2}$.
Given a closed normal o-graph, we replace its vertices with the following diagram to get a $\mathcal{H}(H)$-colored diagram $C_{\Gamma}$.
\begin{figure}[H]
  \centering
  \includegraphics{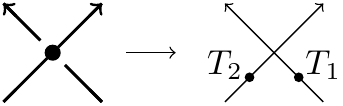}
  \hspace{8mm}
  \includegraphics{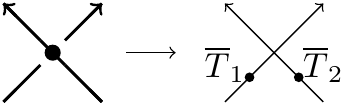}
  \caption{How to associate beads to vertices.}
  \label{fig:beads associated to vertices}
\end{figure}
\noindent
Since closed normal o-graph satisfies axiom \textbf{C1} in the definition \ref{def:bp}, we can perform the moves in Figure $\ref{fig:crm}$, $\ref{fig:bs}$ and slide beads on $C_{\Gamma}$ to get a closed circle with a bead $J_{\Gamma}$ in $\mathcal{H}(H)$.
Because one can permute the beads as in Figure \ref{fig:bs2}, $J_{\Gamma}$ is well-defined in the quotient space $\mathcal{H}(H)/[\mathcal{H}(H),\mathcal{H}(H)]$, which can be identified with $\mathbb{K}$ by Proposition \ref{prop:K}. 

\begin{defi}
 $Z(\Gamma ;\mathcal{H}(H)):=\raisebox{2pt}{$\chi$}_{Fock}(J_{\Gamma})$.
\end{defi}

\begin{ex}The invariant of $S^3=L(1,1)$:

\begin{figure}[H]
  \centering
  \includegraphics{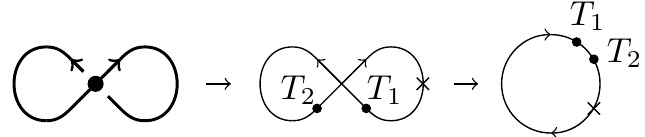}
\end{figure}
\noindent
Thus we have  $Z(S^3;\mathcal{H}(H))=\raisebox{2pt}{$\chi$}_{Fock}(T_2T_1)=\sum_i\raisebox{2pt}{$\chi$}_{Fock}(e^i\otimes e_i)$. By Proposition \ref{prop:fock is int}, we have
\begin{align*}
Z(S^3;\mathcal{H}(H))&=e^i(e_L)\mu_R(e_i)\\
&=\mu_R(e^i(e_L)e_i)\\
&=\mu_R(e_L)\\&=1.
\end{align*}
\end{ex}

\begin{ex}The invariant of $\mathbb{R}P^3=L(2,1)$:

\begin{figure}[H]
  \centering
  \includegraphics{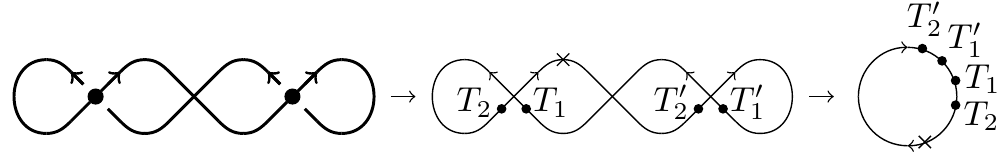}
\end{figure}
\noindent
Thus we have  $Z(\mathbb{R}P^3;\mathcal{H}(H))=\raisebox{2pt}{$\chi$}_{Fock}(T_2T_1T^{\prime}_1T^{\prime}_2)=\sum\limits_{i,j}\raisebox{2pt}{$\chi$}_{Fock}(e^i\cdot(e_{i(1)}e_{j(1)}\rightharpoonup e^j)\otimes e_{i(2)}e_{j(2)})=\mathrm{Tr}(S)$, where $S$ is the antipode.
\end{ex}

\section{Main theorem}
\label{sec:mt}
Recall from Section \ref{Hopf} that a Hopf algebra $H$ is called unimodular if the left cointegrals are also the right cointegrals, and counimodular if the left integrals are also the right integrals.
We prove the following main theorem.
\begin{thm} \label{mt}
Let $H$ be a finite dimensional involutory unimodular counimodular Hopf algebra over $\mathbb{K}$, and $\Gamma$ a closed normal o-graph of a closed oriented 3-manifold $M$.
Then, $Z(\Gamma,\mathcal{H}(H))$ is an invariant of $M$.
\end{thm}

According to Proposition \ref{thm:BP}, in order to prove that $Z(\Gamma,\mathcal{H}(H))$ is an invariant, we need to show that $Z(\Gamma,\mathcal{H}(H))$ is an invariant under planer isotopy and the local moves (the Reidemeister type moves, the 0-2 move, the MP-moves, and the CP-move) of closed normal o-graph described in Figures \ref{fig:RM} to \ref{fig:CP}.
In \cite{S}, it was essentially proved that the value  $Z(\Gamma,\mathcal{H}(H))$ is an invariant under planer isotopy, the Reidemeister type moves, the 0-2 move, and the MP-moves. Here the Reidemeister type moves are nothing but the \textit{symmetry moves} in \cite{S}, the 0-2 move is a special case of the \textit{colored (0,2) moves},  and the MP-moves are obtained by the \textit{colored Pachner (2,3) moves} and the colored (0,2) moves. Even so, since the frameworks are slightly different, we give another proof in this section. The invariance under the CP-move was not observed in \cite{S}, and we give the proof in the following section, where we use a framework of tensor networks.

\begin{proof} [Proof of Theorem \ref{mt}]
We prove that $Z(\Gamma,\mathcal{H}(H))$ is an invariant under planer isotopy and the Reidemeister type moves, the 0-2 move, and the MP-moves. The proof of the invariance under the CP-moves will be shown in Section \ref{sec:tnm}.
\\
\newline
\textbf{Invariance under planer isotopy and the Reidemeister type moves}
It is obvious from the construction that $Z(\Gamma,\mathcal{H}(H))$ is an invariant under planer isotopy and the Reidemeister type moves described in the Figure \ref{fig:RM}.
\\
\newline
\textbf{Invariance under the 0-2 move.}
Let us calculate the local tensor associated to the right hand side of Figure \ref{fig:02}.
\begin{figure}[H]
  \centering
  \includegraphics{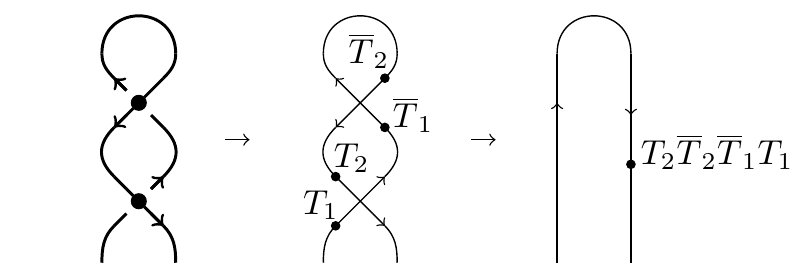}
\end{figure}
\noindent
Thus, we need to show that $T_2\overline{T}_2\overline{T}_1T_1=\epsilon \otimes 1$.
Let $T=\sum\limits_{i}(\epsilon \otimes e_i)\otimes (e^i \otimes 1)$ and $\overline{T}=\sum\limits_{j}(\epsilon \otimes S(e_j))\otimes (e^j \otimes 1)$.
Then,
\begin{align}
T_2\overline{T}_2\overline{T}_1T_1 &= \sum\limits_{i,j}(e^i \otimes 1)(e^j \otimes 1)(\epsilon \otimes S(e_j))(\epsilon \otimes e_i) \\
&=\sum\limits_{i,j}e^ie^j \otimes S(e_j)e_i \label{eq:0-2}
\end{align}
Identifying $H^*\otimes H$ with $\End(H)$, (\ref{eq:0-2}) becomes $x \mapsto S(x_{(2)})x_{(1)}$.
Since $S$ is assumed to be involutive, we have $S(x_{(2)})x_{(1)}=\epsilon(x)1$ .
Thus we have the assertion $T_2\overline{T}_2\overline{T}_1T_1=\epsilon \otimes 1$.
\\
\newline
\textbf{Invariance under the MP moves.}
There are 16 MP moves. Let us write all the calculations down.
\begin{figure}[H]
  \centering
  \includegraphics{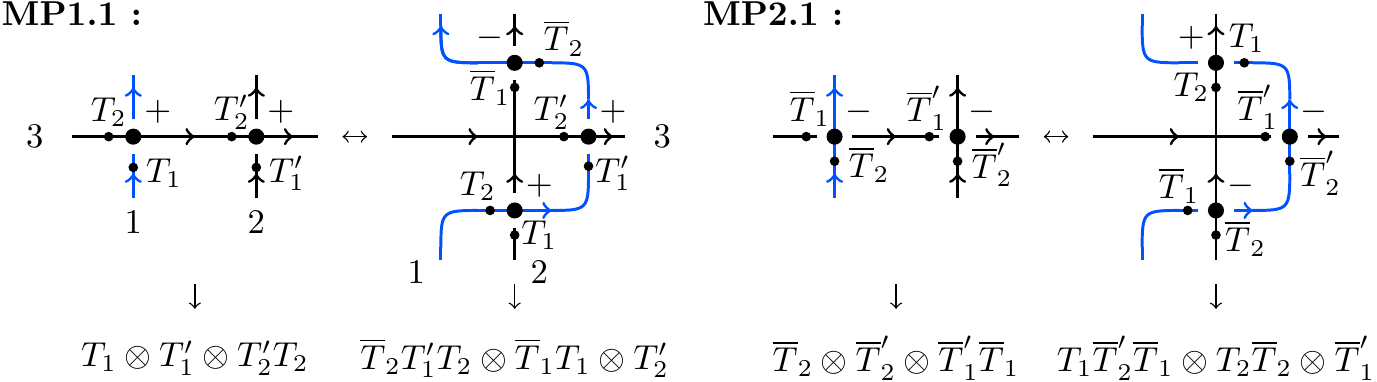}
\end{figure}
\begin{figure}[H]
  \centering
  \includegraphics{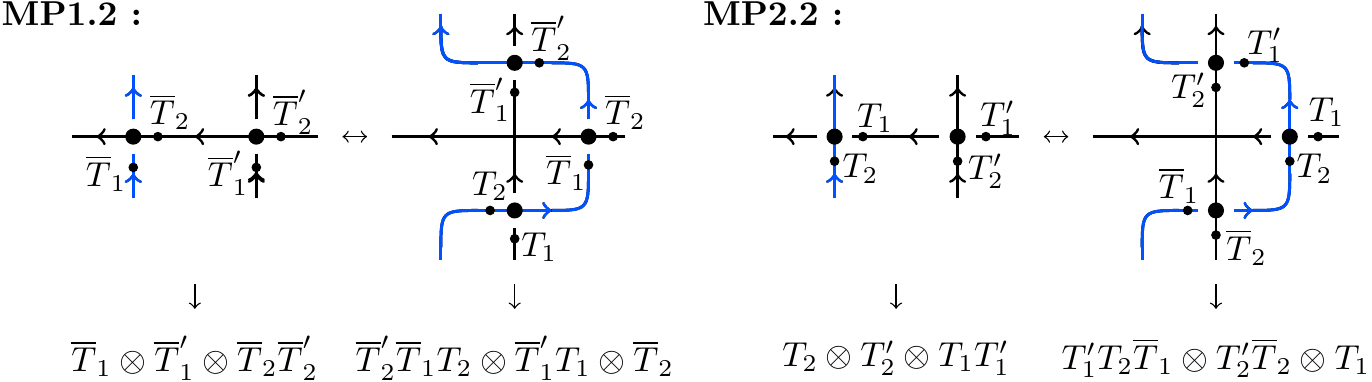}
\end{figure}
\begin{figure}[H]
  \centering
  \includegraphics{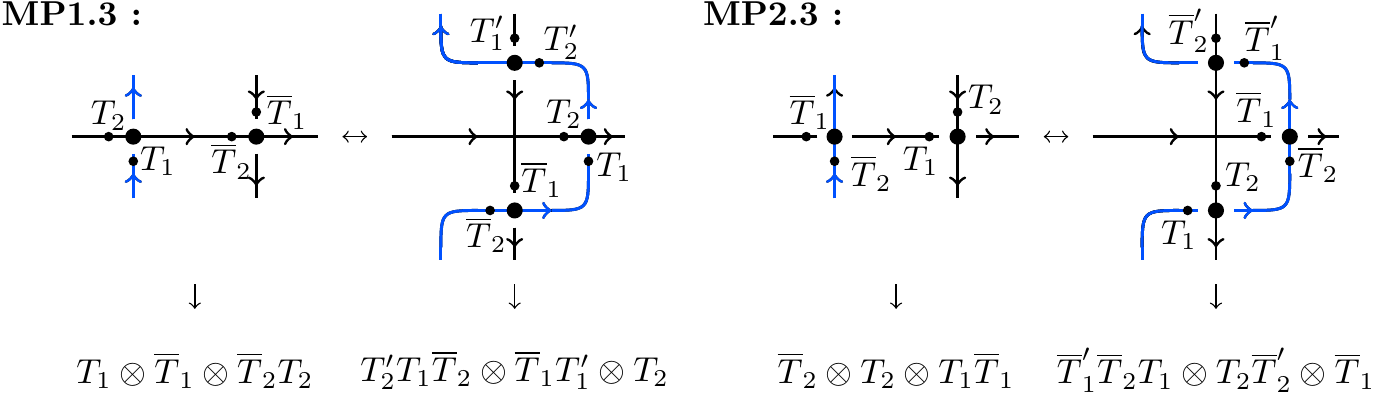}
\end{figure}
\begin{figure}[H]
  \centering
  \includegraphics{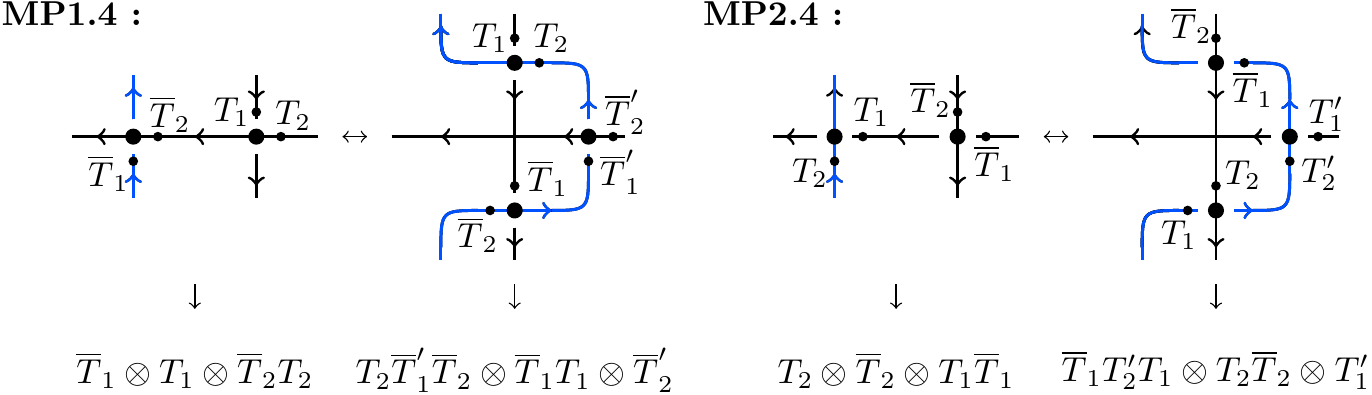}
\end{figure}
\begin{figure}[H]
  \centering
  \includegraphics{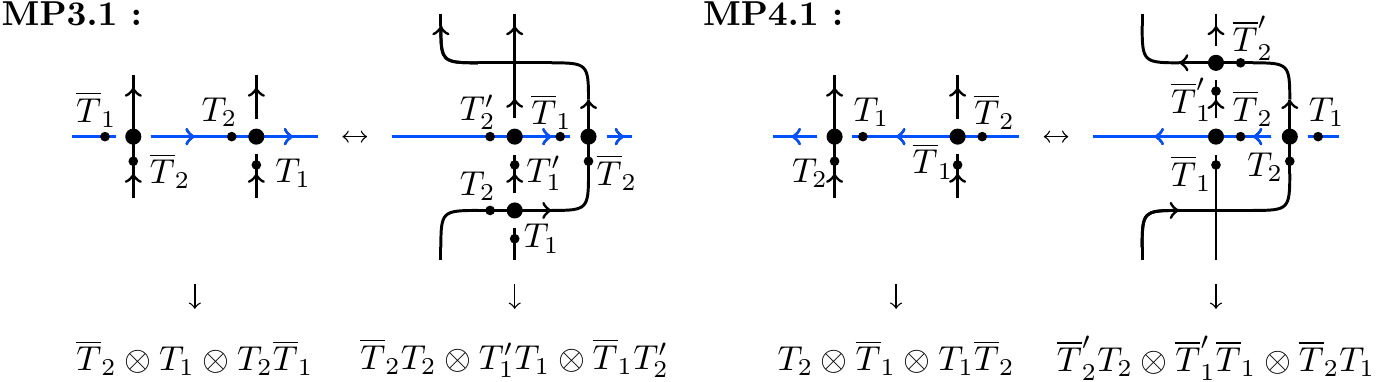}
\end{figure}
\begin{figure}[H]
  \centering
  \includegraphics{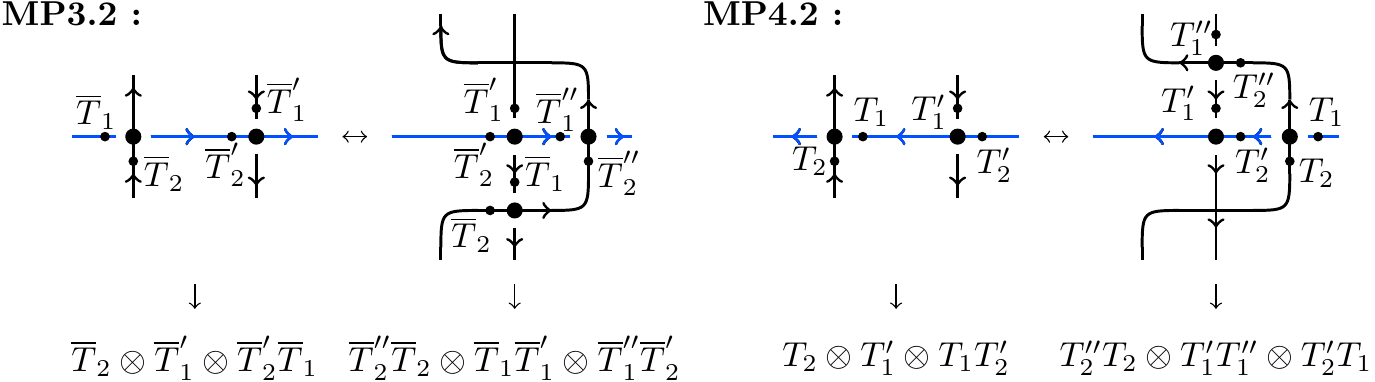}
\end{figure}
\begin{figure}[H]
  \centering
  \includegraphics{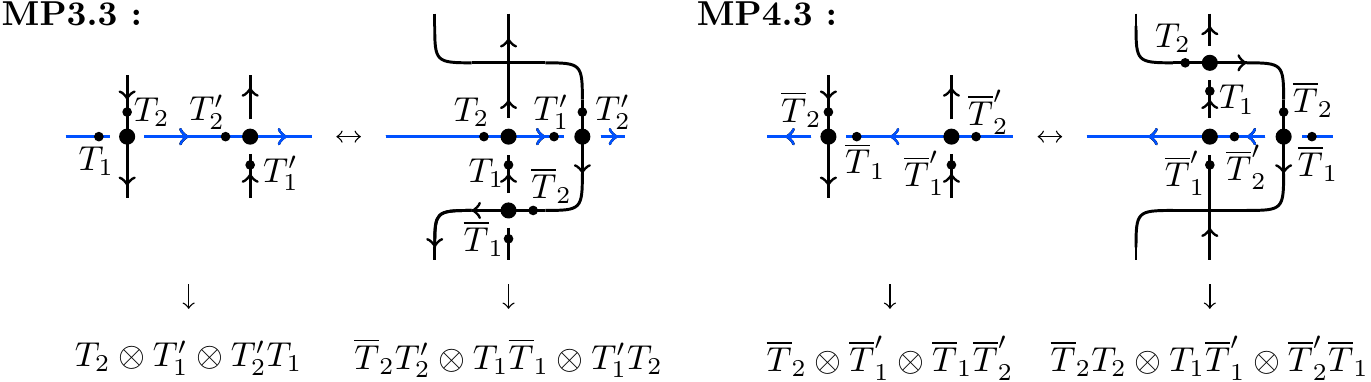}
\end{figure}
\begin{figure}[H]
  \centering
  \includegraphics{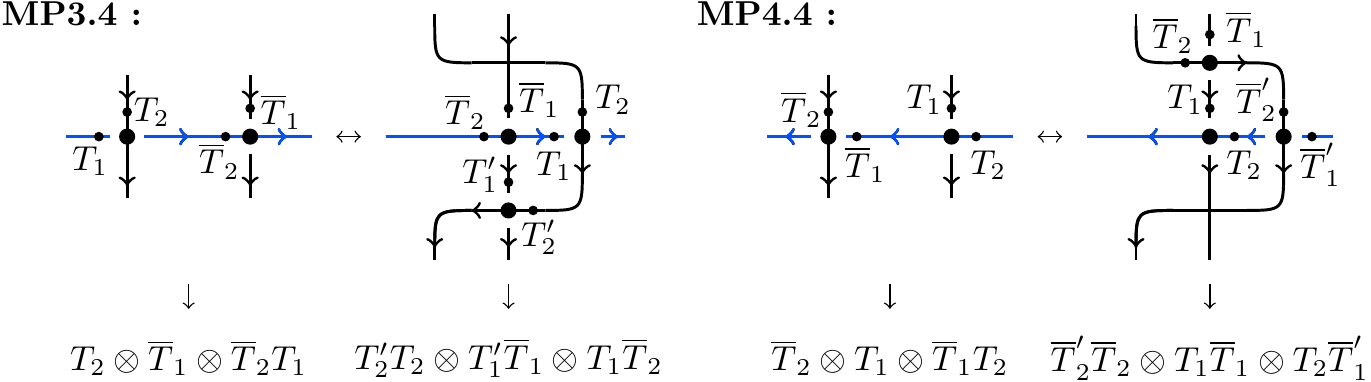}
\end{figure}
Namely, we need to check the following 16 equalities:
\setlength{\itemsep}{2mm}
\setlength{\parskip}{2mm} 
\begin{align*}
\textbf{MP1.1 :}& \quad T_{23}T_{13} = \overline{T}_{21}T_{13}T_{21} &
\textbf{MP2.1 :}& \quad \overline{T}_{32}\overline{T}_{31} = T_{12}\overline{T}_{31}\overline{T}_{12}\\
\textbf{MP1.2 :}& \quad \overline{T}_{13}\overline{T}_{23} = \overline{T}_{21}\overline{T}_{13}T_{21} &
\textbf{MP2.2 :}& \quad T_{31}T_{32} = T_{12}T_{31}\overline{T}_{12}\\
\textbf{MP1.3 :}& \quad \overline{T}_{23}T_{13} = T_2^{\prime}T_1\overline{T}_2 \otimes \overline{T}_1T_1^{\prime}\otimes T_2 & 
\textbf{MP2.3 :}& \quad T_{32}\overline{T}_{31} = \overline{T}_1^{\prime}\overline{T}_2T_1\otimes T_2\overline{T}_2^{\prime}\otimes \overline{T}_1\\
\textbf{MP1.4 :}& \quad \overline{T}_{13}T_{23} = T_2\overline{T}_1^{\prime}\overline{T}_2 \otimes \overline{T}_1T_1\otimes \overline{T}_2^{\prime} &
\textbf{MP2.4 :}& \quad T_{31}\overline{T}_{32} = \overline{T}_1T_2^{\prime}T_1\otimes T_2\overline{T}_2\otimes T_1^{\prime}\\
\\
\textbf{MP3.1 :}& \quad T_{23}\overline{T}_{31} = \overline{T}_{31}T_{23}T_{21} & 
\textbf{MP4.1 :}& \quad T_{31}\overline{T}_{23} = \overline{T}_{21}\overline{T}_{23}T_{31}\\
\textbf{MP3.2 :}& \quad \overline{T}_{23}\overline{T}_{31} = \overline{T}_{31}\overline{T}_{21}\overline{T}_{23} &
\textbf{MP4.2 :}& \quad T_{31}T_{23} = T_{23}T_{21}T_{31}\\
\textbf{MP3.3 :}& \quad T_{23}T_{31} = \overline{T}_2T_2^{\prime}\otimes T_1\overline{T}_1\otimes T_1^{\prime}T_2 &
\textbf{MP4.3 :}& \quad \overline{T}_{31}\overline{T}_{23} = \overline{T}_2T_2\otimes T_1\overline{T}_1^{\prime}\otimes \overline{T}_2^{\prime}\overline{T}_1\\
\textbf{MP3.4 :}& \quad \overline{T}_{23}\overline{T}_{31} = T_{21}T_{31}\overline{T}_{23} & 
\textbf{MP4.4 :}& \quad \overline{T}_{31}T_{23} = T_{23}\overline{T}_{31}\overline{T}_{21}
\end{align*}

Let us see that each equation above is equivalent to the pentagon equation (\ref{eq:pentagon equation}).
Define the map $\tau_{\mathcal{H}}:\mathcal{H}(H)^{\otimes 2}\to\mathcal{H}(H)^{\otimes 2}$ by $\tau_{\mathcal{H}}(x\otimes y)=y\otimes x$ for $x, y\in\mathcal{H}(H)$.
Then, for example, if we multiply $T_{21}$ on \textbf{MP1.1} from left and then apply $\tau_{\mathcal{H}}\otimes \text{id}$, one sees that the result is exactly the pentagon equation.
Similarly, we can reduce \textbf{MP1.2},\textbf{MP2.1},\textbf{MP2.2},\textbf{MP3.1},\textbf{MP3.2},\textbf{MP3.4},\textbf{MP4.1},\textbf{MP4.2}, and \textbf{MP4.4} to the pentagon equation.

Let us see the other 6 equalities.
Define the map $\mathcal{S}:\mathcal{H}(H)\to \mathcal{H}(H)$ by $\mathcal{S}(f\otimes a)=S(f)\otimes S(x)$, where $S$ is the antipode of the Hopf algebra $H$.
Then, for example, we can transform \textbf{MP1.3} into \textbf{MP1.1} by applying $\text{id}\otimes\mathcal{S}\otimes\text{id}$ to the both sides  as follows.
\begin{align*}
(\text{id}\otimes\mathcal{S}\otimes\text{id})(\overline{T}_{23}T_{13}) &= (\text{id}\otimes\mathcal{S}\otimes\text{id})(T_2^{\prime}T_1\overline{T}_2 \otimes \overline{T}_1T_1^{\prime}\otimes T_2)\\
\Leftrightarrow\quad\sum\limits_{i,j}(\epsilon\otimes e_j)\otimes(\epsilon\otimes S^2(e_i))\otimes(e^ie^j\otimes 1)&=\sum\limits_{i,j,k}(e^j(e_{(1)}\rightharpoonup e^k)\otimes e_{(2)})\otimes (\epsilon\otimes S(e_j)S^2(e_k))\otimes(e^i\otimes 1)\\
\Leftrightarrow\quad T_{23}T_{13} &= \overline{T}_{21}T_{13}T_{21},
\end{align*}
where we used the involutivity of $S$ for the last equivalence. Thus \textbf{MP1.3} is also equivalent to pentagon equation. We can show that 
\textbf{MP1.4}, \textbf{MP2.3}, \textbf{MP2.4}, \textbf{MP3.3}, \textbf{MP4.3} are also equivalent to the pentagon equation in similar manners.

Since the pentagon equation holds in the Heisenberg double $\mathcal{H}(H)$, we conclude that $Z$ is an invariant under the MP-moves. 
\\
\newline
\textbf{Invariance under the CP-move.}
This will be proved in Proposition \ref{prop:CP} using tensor network.
\end{proof}

\begin{rem}
The assumption of unimodularity and counimodularity of $H$ will be used only for invariance of the CP-move.
Involutivity was used for the 0-2 move and the MP-moves \textbf{MP1.3}, \textbf{MP1.4}, \textbf{MP2.3}, \textbf{MP2.4}, \textbf{MP3.3} and \textbf{MP4.3}.
The other 10 equalities for the MP-moves do not need any restriction on Hopf algebras to hold.
\end{rem}
\begin{rem}
In \cite{S}*{Theorem 5.1}, there is an error; even for an involutory Hopf algebra, $J$ is not an invariant under the colored Pachner (2,3) move in \cite{S}*{Figure 16}, which cannot be obtained by rotating the allowed one. This excluded move corresponds to the MP-moves with some strands reversed, which are not actually the MP-moves.
\end{rem}

\section{Tensor network approach}
\label{sec:tnm}
In this section, we give a quick review of tensor networks which enable graphical calculus of tensors and linear maps. Then we reformulate the invariant using tensor networks, and prove  the invariance under the CP-move.
\subsection{Tensor network}
\label{ssec:tn}
A \textit{tensor network} over a vector space $V$ is an oriented graph which represents a tensor labeled by the set of open edges, where each incoming (resp. outgoing) edge labels $V^*$ (resp. $V$).
For example, the diagram in Figure \ref{fig:T and F} presents an $(m,n)$ tensor $T\in  (V^*)^{\otimes\mathcal{I}}\otimes V^{\otimes \mathcal{O}}=\mathrm{Hom}(V^{\otimes\mathcal{I}}, V^{\otimes\mathcal{O}})$, where $\mathcal{I}=\{i_1,\ldots, i_m\}$ is the set of incoming edges and $\mathcal{O}=\{o_1,\ldots, o_n\}$ is the set of outgoing edges.

\begin{figure}[ht]
  \centering
  \includegraphics{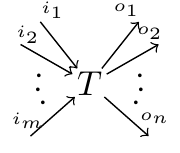}
  \caption{}
  \label{fig:T and F}
\end{figure}

One important feature of tensor networks, which makes this notion practical, is the contraction of tensors.
Given two tensor networks $T$ and $S$, one gets a new tensor network by connecting an outgoing edge $o$ of $T$ and an incoming edge $i$ of $S$ (see Figure \ref{fig:contruction}), which represents the tensor  obtained from $T\otimes S$ by contracting $V_o$ and $(V^*)_i$.
\begin{figure}[ht]
  \centering
  \includegraphics{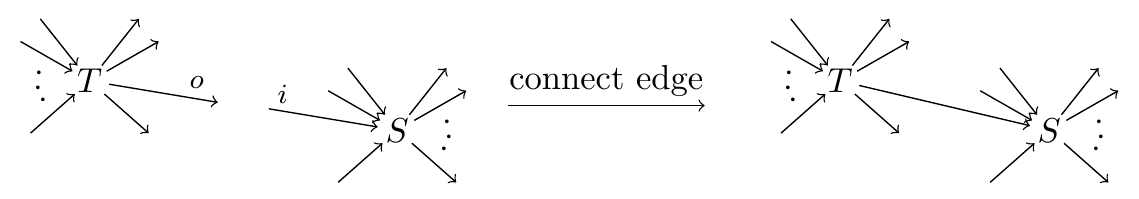}
  \caption{}
  \label{fig:contruction}
\end{figure}

For example, the left diagram in Figure \ref{fig:f and g} represents the composition $g\circ f$ of two maps $f, g\co V\to V$ and the right diagram represents the trace $\sum_{i}f^i_i = \text{tr}(f)
\in \mathbb{K}$ of a map $f\co V\to V$. 

\begin{figure}[H]
  \centering
  \includegraphics{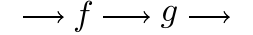}
  \hspace{10mm}
  \includegraphics{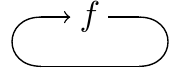}
  \caption{}
    \label{fig:f and g}
\end{figure} 

The axioms of Hopf algebras $(H, M, 1, \Delta, \epsilon, S)$ are represented as follows; 
\begin{figure}[H]
  \centering
  \includegraphics{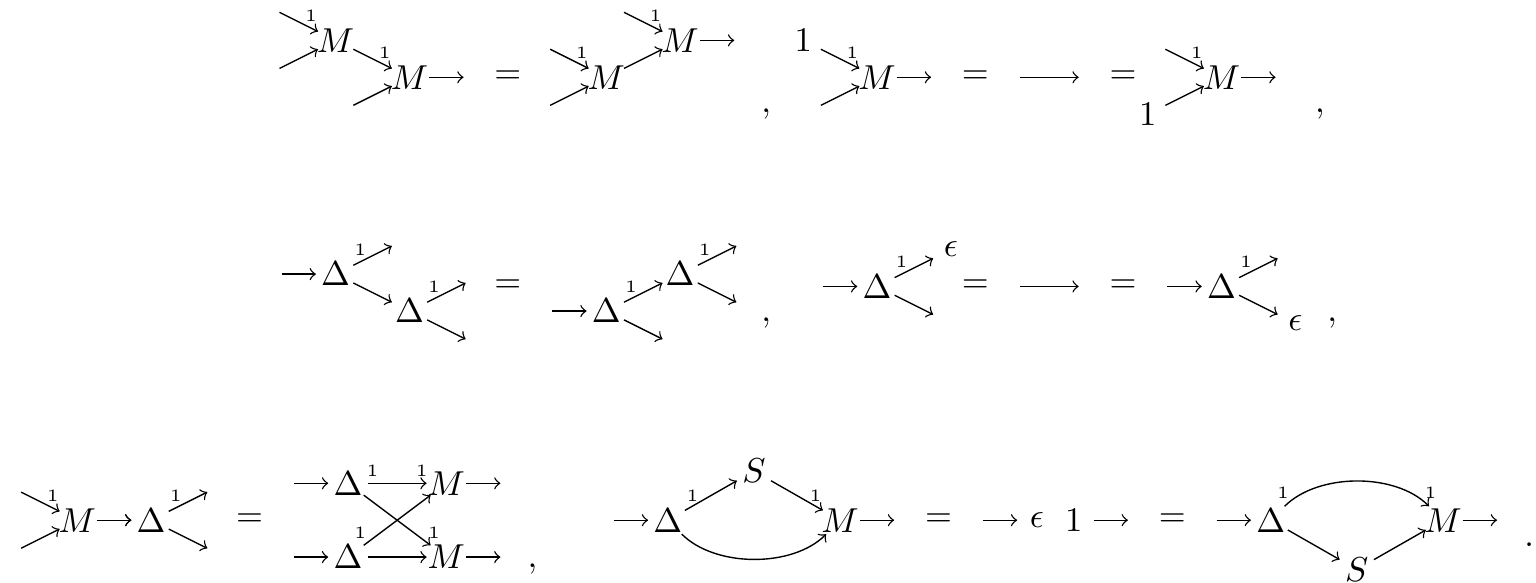}
\end{figure}
\noindent Here, we assume that the incoming edges (resp. outgoing edges) are counted in counter-clockwise
 (resp. clockwise) order from the edge subscripted by $1$, which labels $V^{\otimes m}$ (resp. $(V^*)^{\otimes m}$) in order from the left. 
In obvious cases, the subscript $1$ will be omitted.

\subsection{Reformulation of invariant}
\label{sec:IH}
We assume $H=(H, M, 1, \Delta, \epsilon, S)$ to be a finite dimensional involutory unimodular counimodular Hopf algebra.

An \textit{o-tangle} is an oriented virtual tangle diagram in $[0,1]^2$ such that each boundary point is on the bottom  $[0,1]\times \{0\}$ or on the top $[0,1]\times \{1\}$.
For finite sequences $\varepsilon,\varepsilon'$ of $\pm$, an $(\varepsilon,\varepsilon')$ o-tangle is an o-tangle having boundary points on the bottom and top compatible to $\varepsilon $ and $\varepsilon'$, respectively, where compatible means if an edge is oriented upwards (resp. downwards) then it is connected to $+$ (resp.  $-$). 

Let $\mathcal{T}_{o}$ be the \textit{category of o-tangles}, where objects are finite sequences of $\pm$ including the empty sequence $\emptyset$, and morphisms $\mathcal{T}_{o}(\varepsilon,\varepsilon')$ from  $\varepsilon$ to $\varepsilon'$ are isotopy classes of $(\varepsilon,\varepsilon')$ o-tangles.
As usual, $\mathcal{T}_{o}$ is a strict monoidal category with the unit object $\emptyset$, and the composition $\Gamma' \circ \Gamma$ of $(\varepsilon,\varepsilon')$ o-tangle $\Gamma$ and  $(\varepsilon',\varepsilon'')$ o-tangle $\Gamma'$ is obtained by connecting the $\varepsilon'$ type boundary points on the top of $\Gamma$ to these on the bottom of $\Gamma'$.
We can construct a monoidal functor $Z(\,*\,;H)$ from the category of o-tangle $\mathcal{T}_{o}$ to the category of finite dimensional vector spaces $\text{Vect}_{\mathbb{K}}$ as follows.

For the object $+$ (resp. $-$), we set $Z(+;H):=H^*$ (resp. $Z(-;H):=H$).
For a sequence $\varepsilon=(\varepsilon_1, \ldots,\varepsilon_n)$ in ${\pm}$, let $H^{\varepsilon}$ denote $H^{\varepsilon_1}\otimes \cdots \otimes  H^{\varepsilon_n}$ with $H^+=H^*$ and $H^-=H$. 
To a given $(\varepsilon,\varepsilon')$ o-tangle $\Gamma$, we associate a tensor network over $H$, which represents a linear map;
we replace each positive (resp. negative) crossing of $\Gamma$ with the tensor network as the left (resp. right) picture below, and then we obtain $Z(\Gamma;H)\in \mathrm{Hom}(H^{\varepsilon},H^{\varepsilon'})$ by connecting the boundary points of the above tensor networks following the strands of $\Gamma$. Note that the corresponding strands of $Z(\Gamma;H)$ are oriented in the opposite direction to these of $\Gamma$, and the linear map $Z(\Gamma;H)$ sends a tensor $T\in H^{\varepsilon}$ to the tensor $Z(\Gamma;H)(T)$ which is obtained by concatenating $T$ to the bottom of $Z(\Gamma;H)$.

\begin{figure}[H]
  \centering
  \includegraphics{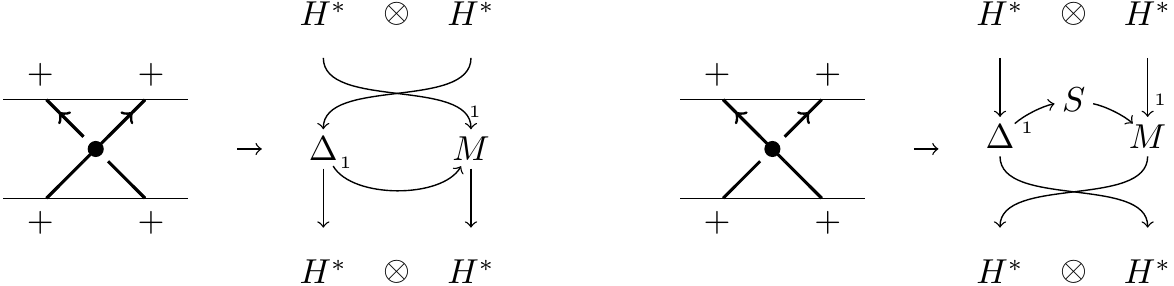}
\end{figure}

Note that a maximum point plays a role of evaluation map, and a minimum point plays a role of coevaluation map as below.

\begin{figure}[H]
    \centering
    \includegraphics{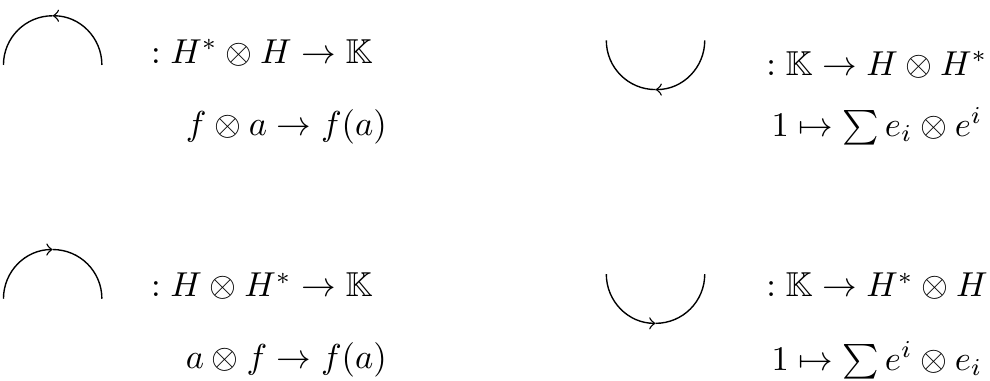}
\end{figure}

Since a closed normal o-graph $\Gamma$ is an $(\emptyset,\emptyset)$ o-tangle, it is sent to an endomorphism $Z(\Gamma;H)$ of $\mathbb{K}$, which is presented by a scalar in $\mathbb{K}$. By abusing the notation we also denote the scalar by $Z(\Gamma;H)\in \mathbb{K}$.

\begin{ex}
For a closed normal o-graph  $\Gamma$ for $S^3$, which is $(\emptyset,\emptyset)$ o-tangle as the left picture below, the resulting tensor network $Z(\Gamma;H)$ is the right picture below.
\begin{figure}[H]
  \centering
  \includegraphics{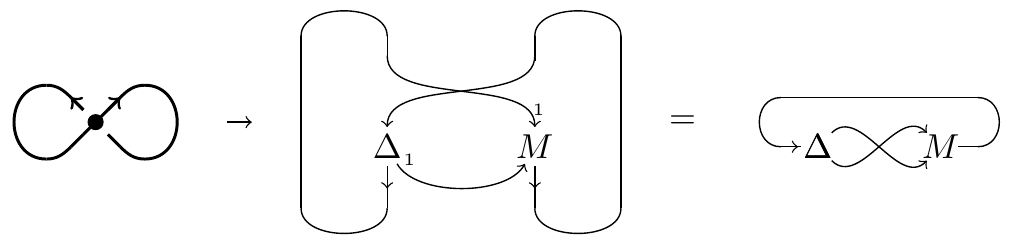}
\end{figure}
Thus we have $Z(\Gamma;H)=\text{tr}(M\circ\tau\circ\Delta)\in\mathbb{K}$, where $\tau(x\otimes y)=y\otimes x$.
\end{ex}
Let $\Gamma$ be a closed normal o-graph and $Z(\Gamma;\mathcal{H}(H))$ the invariant defined in Section \ref{sec:ri}.
\begin{prop}\label{prop:same invariants}
$Z(\Gamma;\mathcal{H}(H))=Z(\Gamma;H)$.
\end{prop}
\begin{proof}
First, let us consider an oriented strand with a bead $f\otimes a \in \mathcal{H}(H)$ as in the left picture below.
We will think of this strand with a bead as the action of $f\otimes a$ on the Fock space $F(H^*)$, which was given by $\phi(f\otimes a)\co g\mapsto f(a\rightharpoonup g)$ for $g \in H^*$, where $f(a\rightharpoonup g)\co x \mapsto f(x_{(1)})g(x_{(2)}a)$ for $x\in H$.
Graphically this map $\phi(f\otimes a)\in \Hom(H^*,H^*)$ can be represented by the tensor network in the right picture below.
\begin{figure}[H]
    \centering
    \includegraphics{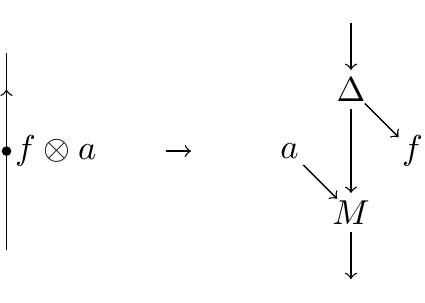}
\end{figure}
\noindent
If there are multiple beads, we replace each bead with the above tensor network. Since each map is an action, the result is well-defined under the move of Figure \ref{fig:bs}.

Then, let us consider the oriented closed curve with a bead $f\otimes a\in\mathcal{H}(H)$.
Recall that $\raisebox{2pt}{$\chi$}_{Fock}(f\otimes a)$ is the trace of the linear map defined by the action of $f\otimes a$.
As remarked in Section \ref{ssec:tn}, in terms of tensor network, taking a trace is just connecting the incoming edge with the outgoing edge.
Thus, we have the following.
\begin{figure}[H]
    \centering
    \includegraphics{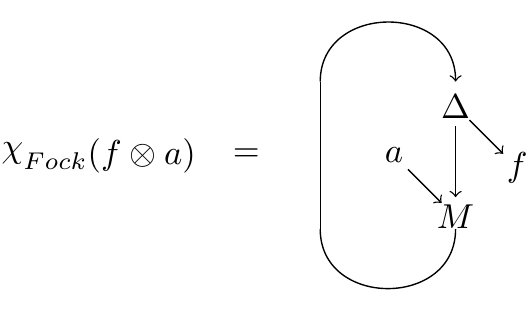}
\end{figure}

Finally, let $\Gamma$ be a closed normal o-graph.
Replacing its vertices as in Figure \ref{fig:beads associated to vertices} and sliding beads, we get a single bead $J_{\Gamma}\in\mathcal{H}(H)$, and the invariant $Z(\Gamma, \mathcal{H}(H))$ was defined as $\raisebox{2pt}{$\chi$}_{Fock}(J_{\Gamma})$.
Here, before the sliding process, we replace each bead with a corresponding tensor network as above, and compare the result to $Z(\Gamma;H)$.
Since the beads only appear at the vertices of $\Gamma$, we just need to look at the associated tensor network for these beads:
\begin{figure}[H]
    \centering
    \includegraphics{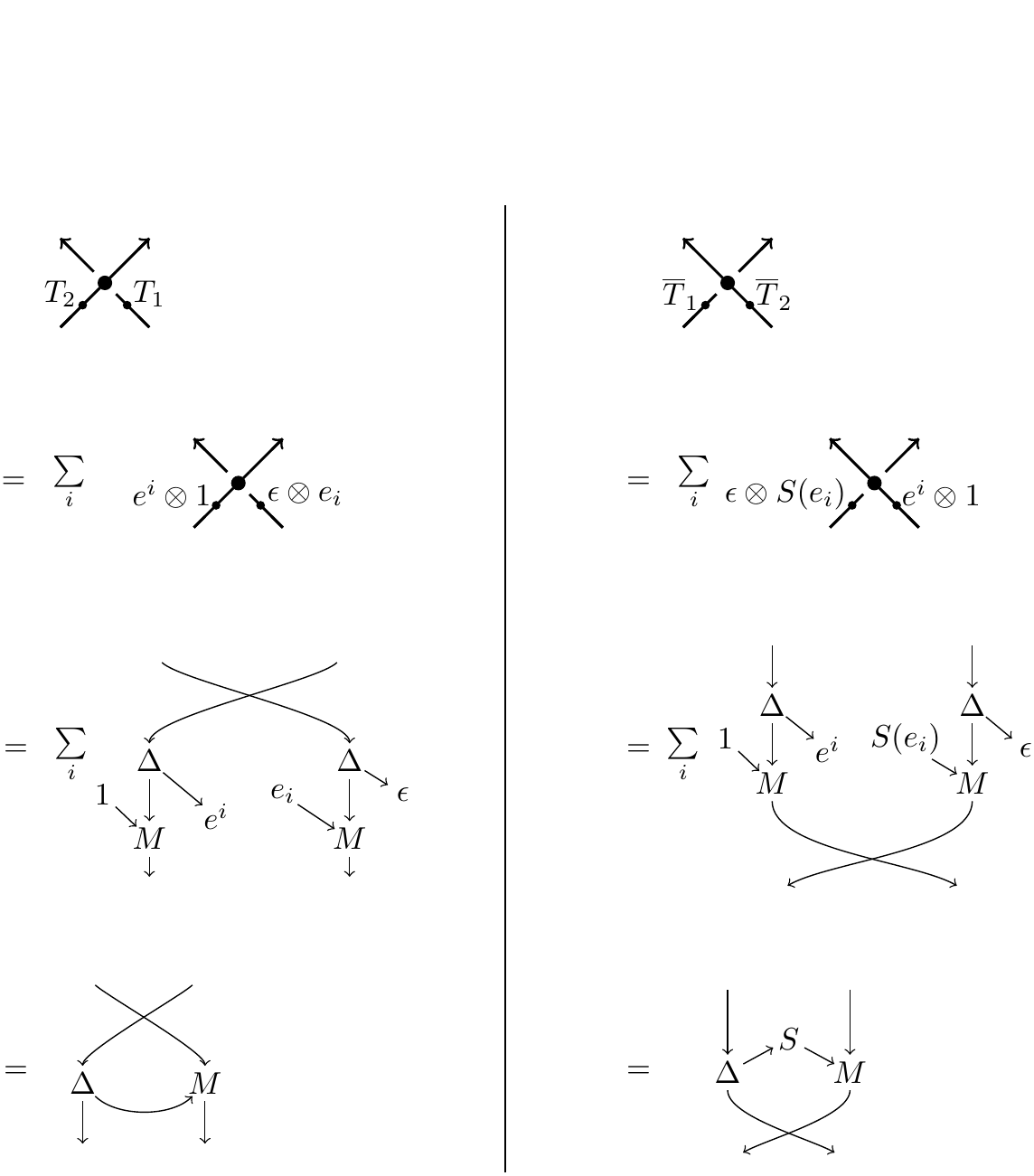}
\end{figure}
\noindent These are the same tensor networks associated with vertices in the definition of $Z(\Gamma;H)$.
Thus we have the assertion.
\end{proof}

\subsection{Invariance under CP-move}
We prove the invariance of $Z(\Gamma;H)$ under the CP-move.  Recall from  Section \ref{sec:hd} for the definition of integrals.
Let $e_R\in H$ and $\mu_L\in H^*$ be a left cointegral and a left integral satisfying $\mu_L(e_L)=1$. 
In terms of tensor networks, we have
\begin{figure}[H]
    \centering
    \includegraphics{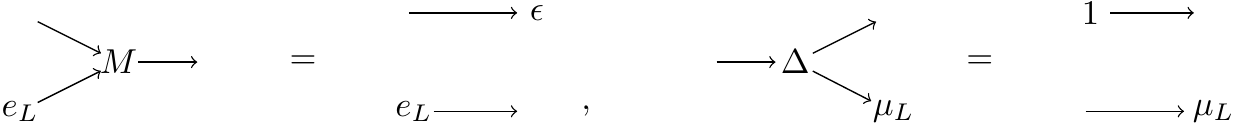}
\end{figure}

\begin{lem}The following equality holds in any Hopf algebra.
\begin{figure}[H]
    \centering
    \includegraphics{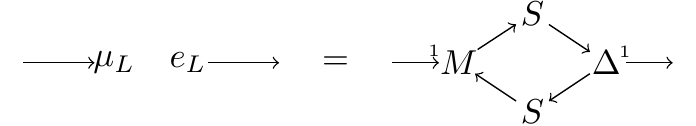}
\end{figure}
\label{lem:integral}
\end{lem}

\begin{proof}
See \cite{Ku2}, Lemma 3.3.
\end{proof}
Set $e_R:=S(e_L)$. Since $S$ is an anti-algebra map, $e_R$ is a (non-zero) right cointegral.

\begin{lem}\label{lem:integral is cyclic}
For a finite dimensional involutory counimodular Hopf algebra, we have
\begin{align*}
\Delta^{\mathrm{op}}(e_R) &= \Delta(e_R),
\end{align*}
where $\Delta^{\mathrm{op}}(x):=x_{(2)}\otimes x_{(1)}$.
\end{lem}

\begin{proof}
See \cite{R}, Theorem 10.5.4.
\end{proof}

\begin{lem}\label{tw}
Let $H$ be an involutory Hopf algebra. We have the following.
\begin{figure}[H]
    \centering
    \includegraphics{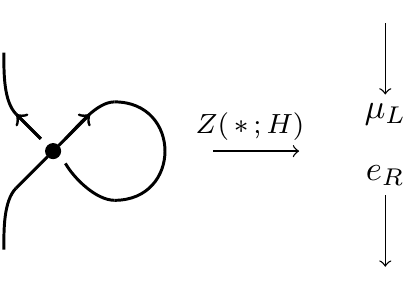}
\end{figure}
\end{lem}

\begin{proof}
Replacing the vertex of the o-tangle with the corresponding tensor network we have
\begin{figure}[H]
    \centering
    \includegraphics{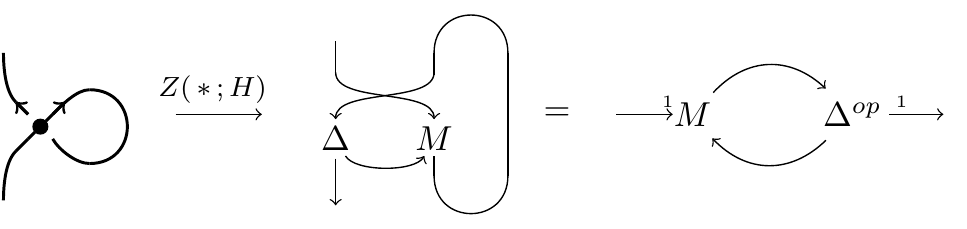}
\end{figure}
Using Lemma \ref{lem:integral} we get
\begin{figure}[H]
    \centering
    \includegraphics{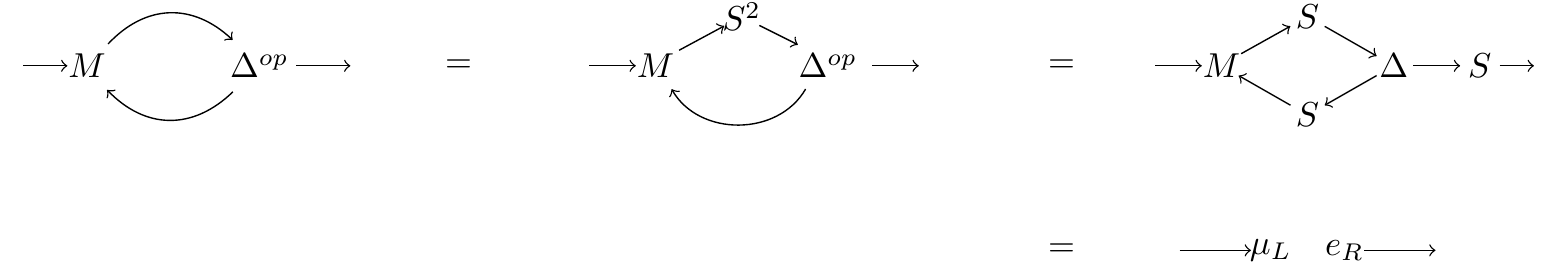}
\end{figure}
\end{proof}

\begin{prop}
\label{prop:CP}
Let $H$ be a finite dimensional involutory unimodular counimodular Hopf algebra.
Then $Z(\Gamma;H)$ is an invariant under the CP-move in Figure \ref{fig:CP}.
\end{prop}

\begin{proof}
First, we evaluate the left hand side of the CP-move.
Using Lemma \ref{tw}, the twists in the closed normal o-graph can be replaced by integrals,  and thus we have
\begin{figure}[H]
    \centering
    \includegraphics{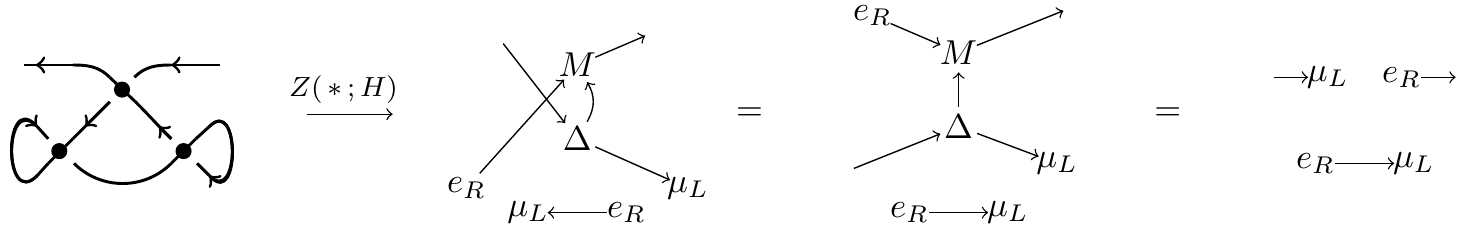}
\end{figure}
\noindent
Next, we evaluate a part of the right hand side of the CP-move.
\begin{figure}[H]
    \centering
    \includegraphics{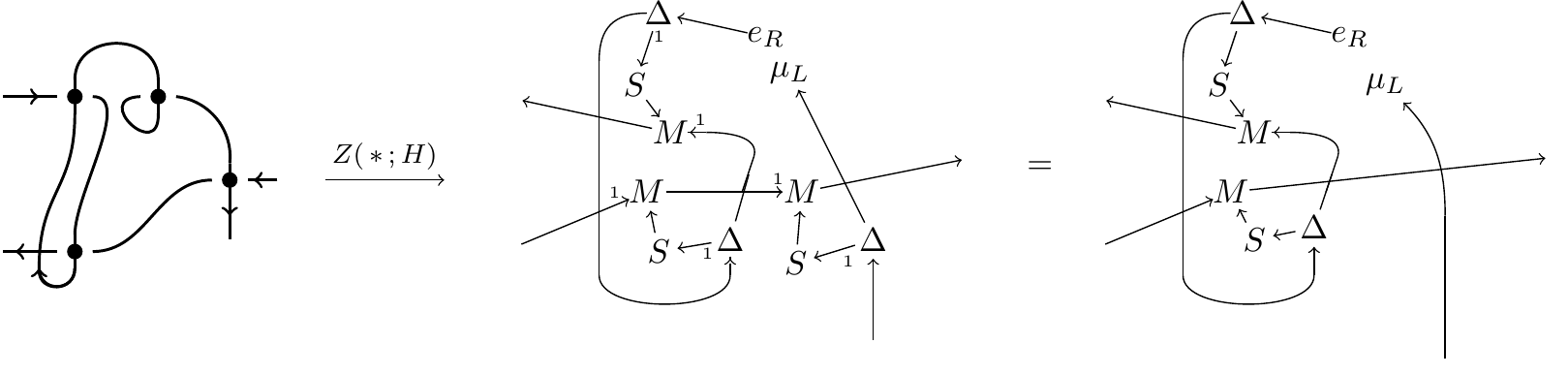}
\end{figure}
\noindent
We used the definition of left integral $\mu_L$ for the equality above.
From Lemma \ref{lem:integral is cyclic}, $e_R$ is cyclic. Thus,
\begin{figure}[H]
    \centering
    \includegraphics{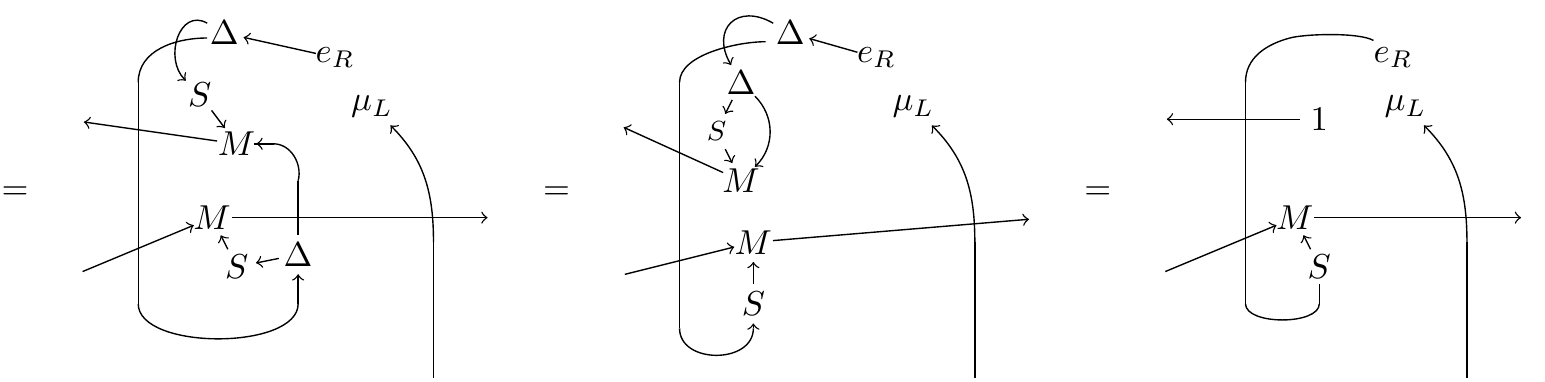}
\end{figure}
\noindent
The second equality follows from coassociativity of the comultiplication, and the last equality follows from  the property of the antipode. 
Thus right hand side of the CP-move becomes
\begin{figure}[H]
    \centering
    \includegraphics{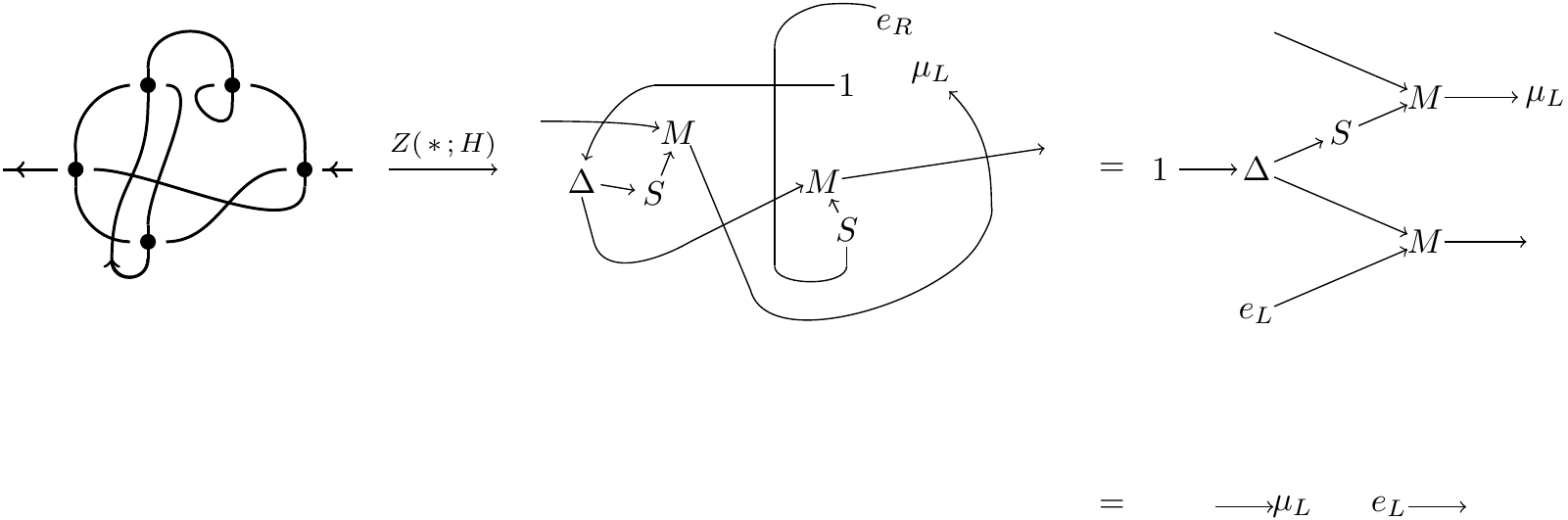}
\end{figure}
\noindent
Finally we need only to show that the following equality holds for involutory unimodular counimodular Hopf algebra.
\begin{figure}[H]
    \centering
    \includegraphics{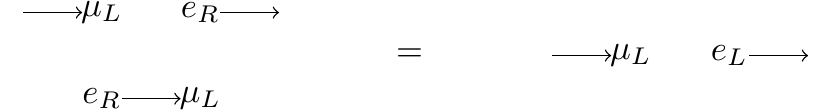}
\end{figure}
\noindent
Since $H$ is unimodular, $e_R=ke_L$ for some non-zero $k \in \mathbb{K}$.
Since $\mu_L(e_L)$ is $1$, $\mu_L(e_R)=k$ and the left hand side of the above tensor network is the same as right hand side times $k^2$.
Applying $S^2$ to $e_L$ and using the fact that $S$ is involutive, we see that $k^2=1$, and the above equality follows.
\end{proof}

\section{Properties}
\label{sec:csf}
We continue to assume that $H$ is a finite dimensional involutory unimodular counimodular Hopf algebra over $\mathbb{K}$.
For a  closed oriented $3$-manifold $M$ and a closed normal o-graph $\Gamma$ representing $M$, set  $Z(M;H):=Z(\Gamma;H)$.

\subsection{Connected sum formula}
For  closed oriented $3$-manifolds $M,N$, let $M\# N$ be the connected sum of them.
\begin{prop}
$Z(M\# N;H)=Z(M;H)Z(N;H)$.
\end{prop}

\begin{proof}
Let $\Gamma_M$  and $\Gamma_N$ be closed normal o-graphs representing $M$ and $N$ respectively.
Let $\Gamma_M\# \Gamma_N$ be the connected sum of closed normal o-graphs defined in Figure \ref{fig:connected sum}.

\begin{figure}[ht]
    \centering
    \includegraphics{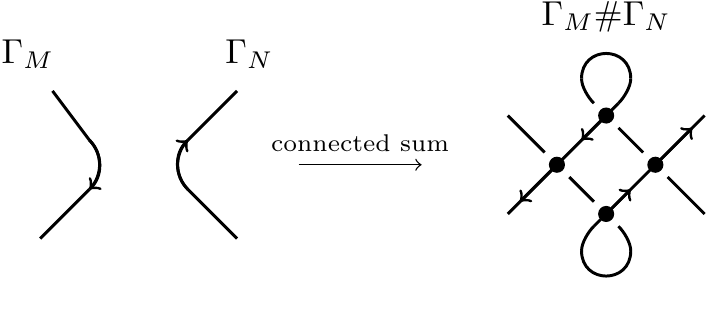}
    \caption{}
    \label{fig:connected sum}
\end{figure}

In \cite{Ko}, Y. Koda showed that for two closed normal o-graphs $\Gamma_M$ and $\Gamma_N$, the $3$-manifold represented by $\Gamma_M\# \Gamma_N$ is $M\# N$. We show the assertion by comparing the tensor networks for the closed normal o-graphs in the left and right hand sides of the Figure \ref{fig:connected sum}. 
Note that 
\begin{figure}[H]
    \centering
    \includegraphics{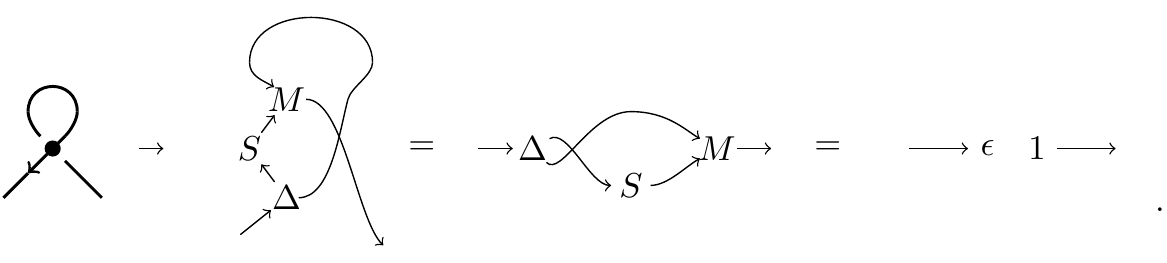}
\end{figure}
Thus we have
\begin{figure}[H]
    \centering
    \includegraphics{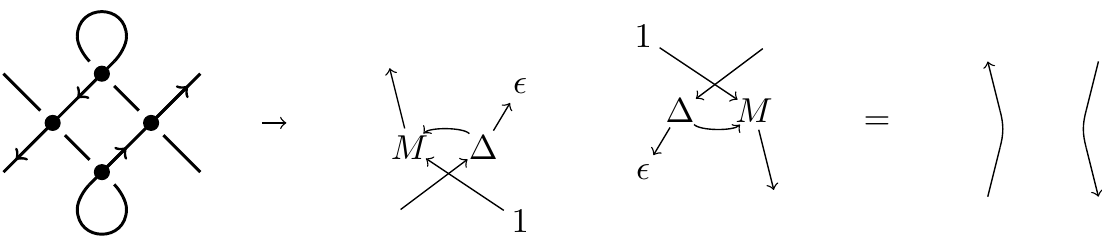}
\end{figure}
which implies the assertion.
\end{proof}

\subsection{Group algebra}
We show that the invariant $Z(M;\mathbb{C}[G])$ with the group algebra $\mathbb{C}[G]$ of a finite group $G$ counts the number $|\Hom(\pi_1M,G)|$ of group homomorphisms from the fundamental group $\pi_1 M$ of $M$ to $G$.
The proof essentially follows the line of [KR].

The algebra $\mathbb{C}[G]$ has a canonical basis given by $\{ g \}_{g\in G}$ and the Hopf algebra structure given by $\Delta(g)=g\otimes g$, $\epsilon(g)=1$, and $S(g)=g^{-1}$.
Note that $\mathbb{C}[G]$ is involutory and also unimodular and counimodular.
The dual group algebra $\mathbb{C}(G):=\mathbb{C}[G]^*$ has the dual basis $\{ \delta_g \}_{g\in G}$, and the dual Hopf algebra structure is given as follows.
\begin{center}
    $\delta_{g}\cdot\delta_{h} = \delta _{g,h}\delta_{g} \qquad 1_{\mathbb{C}(G)} = \sum\limits_{g\in G}\delta_{g}$
\end{center}
\begin{center}
    $\Delta(\delta_{g}) = \sum\limits_{hk=g}\delta_{h}\otimes \delta_{k} \qquad \e(\delta_g) = \delta_g(e)$
\end{center}
\begin{center}
    $S(\delta_g) = \delta_{g^{-1}}$
\end{center}
where  $\delta_{g, h}\in \{ 0,1 \}$ is $1$ if $g=h$ and $0$ otherwise, and $e$ is the unit of $G$.

The left action of $x\in G\subset\mathbb{C}[G]$ on $\delta_{g}\in\mathbb{C}(G)$ is given by
\begin{align*}
x\rightharpoonup \delta_g = \delta_{gx^{-1}}.
\end{align*}

\begin{prop}
For a closed oriented $3$-manifold $M$ and a finite group $G$, we have $Z(M;\mathbb{C}[G])=|\Hom(\pi_1M,G)|$.
\end{prop}

\begin{proof}
Let $\Gamma$ be a closed normal o-graph representing $M$.
In \cites{Ko, Ko3}, Y. Koda gave an explicit formula for the fundamental group $\pi_1 M$  in terms of closed normal o-graph $\Gamma$.  Let $E$ be the set of all edges of $\Gamma$. We consider the group generated by $E$ and consider the relation set $R$ consisting of  $g=l$, $hg=k$ for the edges  $g, l, h, k$ around each vertex as below.
Then the resulting group $\langle E\,|\,R\rangle$ is isomorphic to the fundamental group $\pi_1M$ of $M$,  and the number $|\Hom(\pi_1M,G)|$ is equal to the number of edge colorings  $c\co E \to G$ such that $c(R)$ holds in $G$.
\begin{figure}[H]
    \centering
    \includegraphics{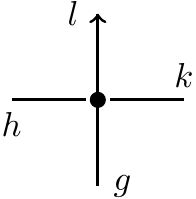}
\end{figure}

We show that the invariant $Z(M;\mathbb{C}[G])$ indeed counts such edge colorings. 
As we explained in Section \ref{sec:IH}, each vertex of closed normal o-graph can be treated as a linear map between $\mathbb{C}(G)^{\otimes 2}$:
\begin{figure}[H]
    \centering
    \includegraphics{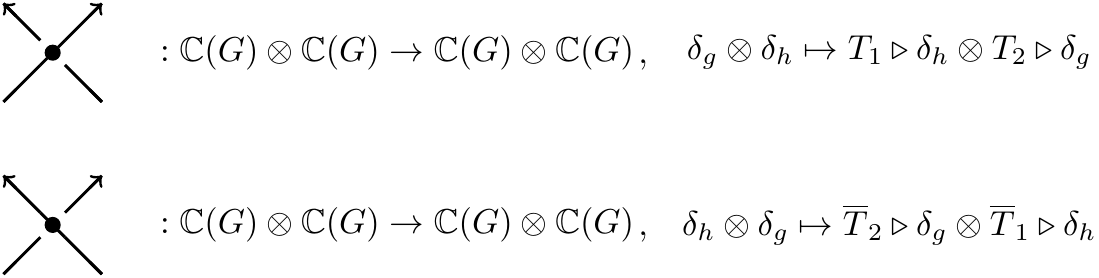}
\end{figure}
\noindent Here we have,
\begin{align*}
T_1\triangleright\delta_h\otimes T_2\triangleright\delta_g &= \sum\limits_{x\in G}(\e\otimes x)\triangleright\delta_h\otimes (\delta_x\otimes e)\triangleright\delta_g \\
&= \sum\limits_{x\in G}x\rightharpoonup\delta_h\otimes\delta_x\cdot\delta_g\\
&= \sum\limits_{x\in G}\delta_{hx^{-1}}\otimes \delta_{x,g}\delta_g\\
&= \delta_{hg^{-1}}\otimes \delta_g\\
\end{align*}
and
\begin{align*}
\overline{T}_2\triangleright\delta_g\otimes\overline{T}_1\triangleright\delta_h &= \sum\limits_{x\in G}(\delta_x\otimes e)\triangleright\delta_g\otimes (\e\otimes x^{-1})\triangleright\delta_h\\
&= \sum\limits_{x\in G}\delta_x\cdot\delta_g\otimes x^{-1}\rightharpoonup\delta_h\\
&= \sum\limits_{x\in G}\delta_{x,g}\delta_g\otimes\delta_{hx}\\
&= \delta_{g}\otimes\delta_{hg}.\\
\end{align*}
We draw these maps as follows.
\begin{figure}[H]
    \centering
    \includegraphics{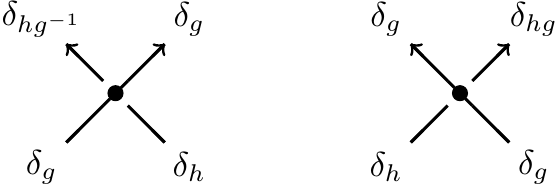}
\end{figure}
\noindent Note that the subscripts of $\delta$ give nothing but an edge coloring (around the vertices) as desired. Furthermore, to connect those vertices by strands means that we insert maximum and minimum points among them. Recall from Section \ref{sec:IH} that maximum and minimum points correspond to evaluations and coevaluations, respectively. In the present case they are the maps shown below, which means after all we sum up all edge colorings.
\begin{figure}[H]
    \centering
    \includegraphics{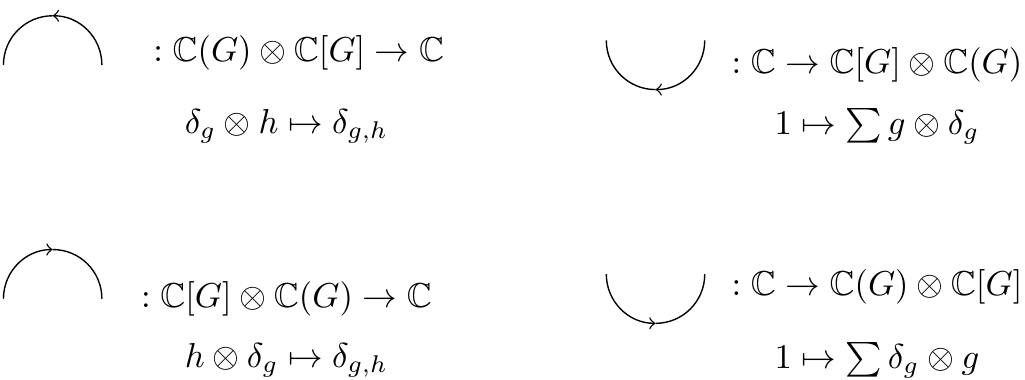}
\end{figure}
\end{proof}

\end{document}